\documentclass[a4paper,12pt,leqno]{amsart}

\usepackage{amscd}
\usepackage{amsmath}
\usepackage{amssymb}
\usepackage{amsthm}
\usepackage{appendix}
\usepackage{bbm}
\usepackage{color}
\usepackage{enumerate}
\usepackage{enumitem}
\usepackage[T1]{fontenc}
\usepackage{geometry}
\usepackage{graphicx}
\usepackage{mathtools}
\usepackage{setspace}
\usepackage{url}
\usepackage[table]{xcolor}

\usepackage{float}
\usepackage{microtype}
\usepackage{verbatim}

\usepackage{tikz}
\usepackage{tikz-cd}
\usetikzlibrary{matrix,arrows,decorations.pathmorphing}
\newcommand{\tikztriangle}{
\begin{tikzpicture}[x=0.3cm,y=0.3cm,baseline=-0.15cm]
\draw (-0.4,-0.433) -- (1.4,-0.433);
\draw (-0.2,-0.7794) -- (0.7,0.7794);
\draw (1.2,-0.7794) -- (0.3,0.7794);
\end{tikzpicture}
}
\newcommand{\tikzsharp}{
\begin{tikzpicture}[x=0.3cm,y=0.3cm,baseline=-0.15cm]
\draw (-0.8,-0.5) -- (0.8,-0.5);
\draw (-0.8,0.5) -- (0.8,0.5);
\draw (-0.5,-0.8) -- (-0.5,0.8);
\draw (0.5,-0.8) -- (0.5,0.8);
\end{tikzpicture}
}
\newcommand{\tikztetrahedron}{
\begin{tikzpicture}[x=0.3cm,y=0.3cm,baseline=-0.18cm]
\draw[dashed] (-0.3,-0.433) -- (1,-0.433);
\draw (-0.3,-0.433) -- (0.4,0.7);
\draw (1,-0.433) -- (0.4,0.7);
\draw (-0.3,-0.433) -- (0.35,-1);
\draw (0.35,-1) -- (1,-0.433);
\draw (0.35,-1) -- (0.4,0.7);
\end{tikzpicture}
}

\usepackage[all,cmtip]{xy}
\def\acts{\ \rotatebox[origin=c]{-90}{$\circlearrowright$}\ }

\usepackage[alphabetic]{amsrefs}
\PassOptionsToPackage{linktocpage}{hyperref}
\usepackage[colorlinks,citecolor=blue]{hyperref}

\newtheorem{prop}{Proposition}[section]

\newtheorem{cor}[prop]{Corollary}
\newtheorem{lem}[prop]{Lemma}
\newtheorem{thm}[prop]{Theorem}
\newtheorem{que}[prop]{Question}
\newtheorem{claim}[prop]{Claim}

\theoremstyle{definition}

\newtheorem{rem}[prop]{Remark}
\newtheorem{notn}[prop]{Notation}

\newtheorem*{claim*}{Claim}

\newtheorem*{notn-term}{Notation and Terminology}

\numberwithin{equation}{section}

\newcommand{\RomanNumeralCaps}[1]{\MakeUppercase{\romannumeral #1}}

\newcommand{\bk}{\mathbf{k}}

\newcommand{\bA}{\mathbb{A}}

\newcommand{\bF}{\mathbb{F}}

\newcommand{\bP}{\mathbb{P}}
\newcommand{\bQ}{\mathbb{Q}}
\newcommand{\bR}{\mathbb{R}}

\newcommand{\bZ}{\mathbb{Z}}

\newcommand{\cE}{\mathcal{E}}
\newcommand{\cF}{\mathcal{F}}

\newcommand{\cL}{\mathcal{L}}
\newcommand{\cO}{\mathcal{O}}
\newcommand{\cQ}{\mathcal{Q}}

\newcommand{\fm}{\mathfrak{m}}

\newcommand{\codim}{\mathrm{codim}}

\newcommand{\Amp}{\mathrm{Amp}}

\newcommand{\Exc}{\mathrm{Exc}}

\newcommand{\rk}{\operatorname{rk}}
\newcommand{\Ext}{\operatorname{Ext}}

\newcommand{\N}{\operatorname{N}}
\newcommand{\NE}{\overline{\operatorname{NE}}}

\newcommand{\Sing}{\operatorname{Sing}}

\newcommand{\Supp}{\operatorname{Supp}}

\title[Endomorphisms of Fano 4-folds]
{Amplified endomorphisms of Fano fourfolds}
\author{Jia Jia, Guolei Zhong}

\address{
\textsc{Department of Mathematics, National University of Singapore, Singapore 119076, Republic of Singapore}
}
\email{jia{\_}jia@u.nus.edu}

\address{
\textsc{Department of Mathematics, National University of Singapore, Singapore 119076, Republic of Singapore.
\newline
\indent
Current Address:
Center for Complex Geometry, Institute for Basic Science, 55 Expo-ro, Yuseong-gu, Daejeon, Republic of Korea, 34126.}
}
\email{zhongguolei@u.nus.edu}
\email{guolei@ibs.re.kr}


\begin{document}

\begin{abstract}
	Let $X$ be a smooth Fano fourfold admitting a conic bundle structure.
	We show that $X$ is toric if and only if $X$ admits an amplified endomorphism;
	in this case, $X$ is a rational variety.
\end{abstract}

\maketitle
\setcounter{tocdepth}{1}
\tableofcontents

\section{Introduction}

We work over an algebraically closed field $\bk$ of characteristic $0$.
As a fundamental building block of non-isomorphic surjective endomorphisms, the rationally connected projective variety plays a significant role in the equivariant minimal model program.

Let $f\colon X\to X$ be a surjective endomorphism on a projective variety.
In a joint work by Meng and the second author, generalizing \cite{fakhruddin2003questions}*{Question 4.4}, they asked the following question (cf.~\cite{meng2020rigidity}*{Question 1.2}), which characterizes toric varieties from dynamical viewpoints.
Recall that $f$ is $q$-\textit{polarized} if $f^*H \sim qH$ for some ample Cartier divisor $H$ on $X$ and integer $q > 1$, \textit{int-amplified} if $f^*L - L$ is ample for some ample Cartier divisor $L$ on $X$, and \textit{amplified} if $f^*L - L$ is ample for some (not necessarily ample) Cartier divisor $L$ on $X$ (cf.~\cite{krieger2017cohomological}); see \cite{meng2018building}*{Propositions 1.1, 2.9} and \cite{meng2020building}*{Theorem 1.1} for the equivalent definitions.
Clearly, ``polarized'' $\Rightarrow$ ``int-amplified'' $\Rightarrow$ ``amplified''.

\begin{que}\label{main-ques}
	Let $X$ be a rationally connected smooth projective variety.
	Suppose that $X$ admits an int-amplified (or polarized) endomorphism $f$.
	Is $X$ a toric variety?
\end{que}

For the surface case, Nakayama confirmed Sato's conjecture that a rational smooth projective surface admitting a non-isomorphic endomorphism is toric (cf.~\cite{nakayama2002ruled}*{Theorem 3});
hence Question~\ref{main-ques} is also considered as a higher dimensional analogue of Sato's conjecture.
Besides, Question~\ref{main-ques} is known to be true under the further assumption that $f$ has totally invariant ramifications (cf.~\cite{meng2020rigidity}*{Theorem 1.4}, \cite{meng2019characterizations}*{Corollary 1.4} and \cite{hwang2011endomorphisms}*{Theorem 1.2}).
Moreover, Meng, Zhang and the second author recently answered Question~\ref{main-ques} affirmatively for smooth Fano threefolds (cf.~\cite{meng2020nonisomorphic}*{Theorem 1.4}):

\begin{thm}\label{thm-threefold-class}
	Let $Y$ be a smooth Fano threefold admitting an (int-)amplified endomorphism $g$.
	Then $Y$ is toric.
	Further, after iteration, either of the following holds.
	\begin{enumerate}[label=\textbf{(\Alph*)}]
		\item If $\rho(Y) \leq 2$, then $Y$ is either $\bP^3$, a (toric) splitting $\bP^1$-bundle over $\bP^2$, or a (toric) blow-up of $Y':=\bP^3$ along a $(g|_{Y'})^{-1}$-invariant line (cf.~\cite{meng2020nonisomorphic}*{Theorem 5.1}).
		\item If $\rho(Y) \geq 3$, then $Y$ admits a conic bundle over a toric surface $Z$ which factors as $Y\xrightarrow{\varphi}Y'\xrightarrow{p_0}Z$ where $\varphi$ is a (toric) blow-up of a (not necessarily Fano) splitting $\bP^1$-bundle $Y'$ over $Z$ along disjoint curves which are intersections of $(g|_{Y'})^{-1}$-invariant prime divisors (cf.~\cite{meng2020nonisomorphic}*{Theorems 6.1 and 8.1}).
	\end{enumerate}
\end{thm}

\begin{rem}
	Note that, when $X$ is a smooth Fano variety, the surjective endomorphism $f$ on $X$ being amplified is equivalent to being int-amplified (cf.~\cite{meng2020nonisomorphic}*{Lemma 2.6}).
\end{rem}


In this paper, we shall give a positive answer to Question~\ref{main-ques} for smooth Fano fourfolds admitting (possibly non-elementary) conic bundles (cf.~Notation~\ref{notation-conic-bundle}).

Theorem~\ref{thm-main} and Corollary~\ref{cor-splitting} below are our main results.
\begin{thm}\label{thm-main}
	Let $X$ be a smooth Fano fourfold admitting a (possibly non-elementary) conic bundle.
	Then the following are equivalent.
	\begin{enumerate}
		\item $X$ is a toric variety.
		\item $X$ admits a polarized endomorphism.
		\item $X$ admits an (int-)amplified endomorphism.
	\end{enumerate}
	In particular, $X$ is rational if one of the above equivalent conditions holds.
\end{thm}

Note that conic bundle structures $\tau$ appear naturally in the study of birational geometry.
Indeed, a Fano contraction of a smooth projective variety with one-dimensional fibres is necessarily a conic bundle (cf.~Lemma~\ref{lem-blowup-or-conic}).
Precisely, when $\dim X=3$, such $\tau$ takes a significant role in the classification of Fano threefolds (cf.~\cite{mori1983fano});
when $\dim X=4$, such $\tau$ characterizes Fano fourfolds with Lefschetz defect $\delta_X\geq 3$ (cf.~Notation \ref{lefschetz}), with all of such fourfolds being rational (cf.~\cite{romano2019characterization}*{Theorem 1.1 and Corollary 1.3}).
In higher dimension, we expect that the Fano contraction is equidimensional under the dynamical assumption, and hence it becomes a conic bundle if the relative dimension is one (cf.~\cite{mori1982threefolds}*{Theorem 3.5} for the threefold case).

Corollary \ref{cor-splitting} below is a generalization of \cite{meng2020nonisomorphic}*{Theorem 6.4}, and it follows from Theorem~\ref{thm-main}.
In contrast, \cite{meng2020nonisomorphic}*{Theorem 6.4} is the main ingredient in the proof of Theorem~\ref{thm-threefold-class} (cf.~\cite{meng2020nonisomorphic}*{Proof of Theorem 8.1}).

\begin{cor}\label{cor-splitting}
	Let $f \colon X \to X$ be an (int-)amplified endomorphism of a smooth Fano fourfold.
	Suppose that $X$ admits a conic bundle $\tau \colon X \to Y$ which factors as $X \xrightarrow{\pi} W \xrightarrow{\tau_0} Y$ where $\pi$ is a composition of blow-ups along disjoint smooth projective surfaces and $\tau_0$ is an elementary conic bundle (cf.~Notation~\ref{notation-elementary}).
	Then $W$ is a splitting $\bP^1$-bundle over $Y$ (cf.~Notation~\ref{notation-alg-spl-bundle}).
\end{cor}

We briefly explain the strategy for the proof of Theorem~\ref{thm-main} $(3) \Rightarrow (1)$.
Let $\tau \colon X \to Y$ be a conic bundle.
By Theorem~\ref{thm-structure-mmp}, $\tau$ factors through an elementary conic bundle $\tau_0 \colon W \to Y$.
If one of the following holds: (1) $\rho(X) - \rho(Y) \neq 2$;
(2) $Y \cong \bP^3$ or the blow-up of a line on $\bP^3$;
or (3) $Y$ is a $\bP^1$-bundle over $\bP^2$, $\bF_0$ or $\bF_1$ (cf.~Notation~\ref{notation-Hirzebruch}), then $W$ is a splitting $\bP^1$-bundle.
Otherwise, after replacing $\tau$ by a new conic bundle $\widehat{\tau}$ if necessary, our new $\widehat{\tau}_0$ is a splitting $\bP^1$-bundle (cf.~Theorem~\ref{thm-splitting}).
Choosing a suitable equivariant minimal model program for $\tau$, one can verify that $X$ is a toric blow-up of a splitting $\bP^1$-bundle over $Y$;
hence $X$ is toric (cf.~\cite{meng2020nonisomorphic}*{Proposition 2.9} and Theorem~\ref{thm-main-int-to-toric}).

\par \vskip 1pc 

\begin{rem}[Difference with previous papers]
	\label{rem-difference}
	In the joint work \cite{meng2020nonisomorphic} of the second author, the Fano threefold case was dealt with by using the known surface theory and the important paper \cite{mori1983fano} for threefolds.
	The Fano fourfold case here is considerably harder.
	First, when showing the smoothness of elementary conic bundles, the discriminant and its self-intersection are less known in higher dimension (cf.~Lemma~\ref{lem-no-3-comp}). 
	Second, when $Y$ is imprimitive, the induced contraction $W \to W'$ may not be $K_W$-negative and $W'$ is possibly not $\bQ$-factorial (cf.~Remark~\ref{rem-reduction-lem}); hence a suitable new conic bundle is required (cf.~Theorem~\ref{thm-splitting}).
	Finally, when proving the splitting-ness of $\bP_Y(\cE) \to Y$, previous results (for surfaces) on ``walls'' do not work any more (cf.~\cite{qin1993equivalence}).
	So we introduce the generalized tool inspired by \cite{greb2017compact} (cf.~Lemmas~\ref{lem-locally-finite-walls}, \ref{lem-semistable-split}).
	Accordingly, more investigations for Fano threefolds are needed (cf.~Claims~\ref{claim-nonempty-wall-P2} $\sim$ \ref{claim-nonempty-wall-F1} and \ref{claim-nonempty-wall-P3}).
\end{rem}



%
%

\par \noindent
\textbf{Acknowledgements.}
The authors would like to thank Professor De-Qi Zhang and Doctor Sheng Meng for many inspiring discussions.
The authors would also like
to thank the referee for the very careful reading and suggestions to improve the paper.
The authors are supported by President's Scholarships of NUS.

\section{Preliminaries}

\begin{notn}\label{notation}
	Let $X$ be a projective variety.
	We use the following notation throughout this paper.
	\begin{enumerate}[label=\arabic*),ref=2.1\,(\arabic*),leftmargin=1.5em]
		\item The symbols $\sim$ (resp.~$\equiv$) denote the \textit{linear equivalence} (resp.~\textit{numerical equivalence}) on $\bQ$- (or $\bR$-) Cartier divisors.

		\item 
		      Let $\N^1(X)$ be the space of $\bR$-Cartier divisors modulo numerical equivalence $\equiv$, and $\rho(X) \coloneqq \dim_{\bR}\N^1(X)$ the \textit{Picard number} of $X$.
		      Let $\N_1(X)$ be the dual space of $\N^1(X)$ consisting of $1$-cycles, $\Amp(X)$ the cone of \textit{ample divisors} in $\N^1(X)$ and $\NE(X)$ the \textit{Mori cone} of pseudo-effective $1$-cycles in $\N_1(X)$.

		\item Let $f \colon X \to X$ be a surjective endomorphism.
		      A subset $D \subseteq X$ is \textit{$f^{-1}$-invariant} (resp.~\textit{$f^{-1}$-periodic}) if $f^{-1}(D) = D$ (resp.~$f^{-s}(D) = D$ for some $s \geq 1$).

		\item A surjective endomorphism $f \colon X \to X$ is \textit{amplified} if $f^*L - L$ is ample for some (not necessarily ample) Cartier divisor $L$ on $X$, \textit{int-amplified} if $f^*L - L$ is ample for some ample Cartier divisor $L$ on $X$, and \textit{polarized} if $f^*H \sim qH$ for some ample Cartier divisor $H$ on $X$ and integer $q > 1$;
		      see \cite{krieger2017cohomological}, \cite{meng2020building} and \cite{meng2018building}.

		\item A normal projective variety $X$ is of \textit{Fano type}, if there is an effective Weil $\bQ$-divisor $\Delta$ on $X$ such that the pair $(X, \Delta)$ has at worst klt singularities and $-(K_X + \Delta)$ is an ample $\bQ$-Cartier divisor.
		      If $\Delta = 0$, we say that $X$ is a \textit{(klt) Fano variety}.

		\item A smooth Fano surface is usually called a \textit{del Pezzo surface}.
		      The \textit{degree} of a del Pezzo surface is defined as the self-intersection number of its canonical divisor.

		\item \label{notation-Hirzebruch}
		      Denote by $\bF_d \coloneqq \bP_{\bP^1}(\cO \oplus \cO(-d))$ the Hirzebruch surface of degree $d$ with $d \in \bZ_{\geq 0}$.

		\item A smooth Fano threefold is \textit{imprimitive} if it is isomorphic to the blow-up of another smooth Fano threefold along a smooth irreducible curve (cf.~\cite{mori1983fano}*{Definition 1.3}).

		\item A normal projective variety $X$ of dimension $n$ is a \textit{toric variety} if $X$ contains a \textit{big torus} $T = (k^*)^n$ as an (affine) open dense subset such that the natural multiplication action of $T$ on itself extends to an action on the whole variety.
		      In this case, $B \coloneqq X \setminus T$ is a divisor;
		      the pair $(X, B)$ is said to be a \textit{toric pair}.

		\item Let $\pi \colon X \to W$ be the blow-up of a smooth toric variety $W$ along a smooth closed subvariety $S$.
		      We say that $\pi$ is a \textit{toric blow-up} if there exists some big torus $T$ acting on $W$ with $T(S) = S$.
		      In this case, $X$ is still toric.

		\item \label{notation-conic-bundle}
		      A fibration $\tau \colon X \to Y$ of smooth projective varieties is a \textit{(regular) conic bundle} if every fibre of $\tau$ is a conic, i.e., it is isomorphic as a scheme to the zeros of a nontrivial section $\cO_{\bP^2}(2)$.
		      Note that \(\tau\) is flat since both \(X\) and \(Y\) are smooth and \(\tau\) is equi-dimensional.
		      If $X$ is further assumed to be Fano, then $\tau$ is a \textit{Fano conic bundle}.
		      \begin{enumerate}[label=(\alph*),ref=2.1\,(11\,\alph*)]
			      \item \label{notation-discriminant}
			            Denote by $\Delta_{\tau} \coloneqq \{y \in Y \mid \tau \text{ is not smooth over } y \}$ the \textit{discriminant} of $\tau$, which is a reduced divisor on $Y$ (cf.~\cite{sarkisov1983conic}*{Proposition 1.8 and Corollary 1.9}).
			      \item There is a double cover $\sigma \colon \widetilde{\Delta_{\tau}} \to \Delta_{\tau}$, which is {\'e}tale over the regular locus of $\Delta_{\tau}$ (cf.~\cite{beauville1977varietes}*{Proposition 1.5} or \cite{sarkisov1983conic}*{\S 1.17}).
			      \item \label{notation-elementary}
			            A conic bundle $\tau \colon X \to Y$ is said to be \textit{elementary}, if $\rho(X) = \rho(Y) + 1$; in this case, the double cover $\sigma$ in (b) is nontrivial over each irreducible component $D_i \subseteq \Delta_{\tau}$;
			            hence $\sigma^{-1}(D_i)$'s are all connected (cf.~\cite{sarkisov1983conic}*{\S 1.17}).
			      \item Let $\cE$ be a locally free sheaf of rank $3$ on a smooth projective variety $Y$, and $\pi \colon \bP_Y(\cE) \to Y$ the standard projection.
			            An irreducible reduced (possibly singular) divisor $V$ such that the generic fibre of $\pi|_V$ is a smooth rational curve is called an \textit{embedded conic} over $Y$ (cf.~\cite{sarkisov1983conic}*{Definition 1.1 and \S 1.5}).
		      \end{enumerate}

		\item \label{notation-alg-spl-bundle}
		      A fibration $\tau \colon X \to Y$ is a (smooth) \textit{$\bP^1$-bundle}, if $\tau$ is a conic bundle and $\Delta_{\tau} = \emptyset$.
		      A fibration $\tau \colon X \to Y$ is an \textit{algebraic $\bP^1$-bundle}, if $X \cong \bP_Y(\cE)$ for some locally free sheaf $\cE$ of rank $2$ on $Y$.
		      An algebraic $\bP^1$-bundle $X \coloneqq \bP_Y(\cE) \xrightarrow{\tau} Y$ is a \textit{splitting $\bP^1$-bundle} if $\mathcal{E}$ is a direct sum of two invertible sheaves.

		\item \label{lefschetz}For a prime divisor $D$ on $X$, the inclusion $i \colon D \hookrightarrow X$ induces a pushforward of $1$-cycles $i_* \colon \N_1(D) \to \N_1(X)$ with the image a linear subspace.
		      The \textit{Lefschetz defect} $\delta_X$ is defined as $\delta_{X} \coloneqq \max \{\codim_{\N_1(X)} i_*(\N_1(D)) \mid D \subseteq X \text { is a prime divisor}\}$ (cf.~\cite{casagrande2012picard}).

		\item \label{notation-walls}
		      Fix a locally free sheaf $\cE$ of rank $2$ on a smooth Fano threefold $Y$.
		      For every saturated subsheaf $\cF \subseteq \cE$, i.e., $\cE/\cF$ is torsion free (so $\cF$ is reflexive and thus locally free), denote by $\xi_{\cF} \coloneqq 2c_1(\cF) - c_1(\cE)$ and define the set
		      \[
			      W_{\cE}(\cF) \coloneqq \{H^2 \mid H \in \Amp(Y), H^2 \cdot \xi_{\cF} = 0\} \subseteq P(Y) \coloneqq \{H^2 \mid H \in \Amp(Y)\}.
		      \]
		      to be a \textit{generalized wall} of $P(Y) \subseteq \N_1(Y)$ with respect to $\cE$ (cf.~\cite{greb2017compact}*{Section 6}).
	\end{enumerate}
\end{notn}

In the rest of this section, we gather several results to be used frequently in the subsequent sections.
We begin with Lemmas~\ref{lem-equiv-MMP} $\sim$ \ref{thm-bh}, which are on endomorphisms.

\begin{lem}\label{lem-equiv-MMP}
	Let $f \colon X \to X$ be a surjective endomorphism on a normal projective variety.
	Then any finite sequence of a minimal model program starting from $X$, is $f$-equivariant (after iteration), if either of the following conditions is satisfied.
	\begin{enumerate}
		\item The (closed) Mori cone $\NE(X)$ has only finitely many extremal rays (this holds when $X$ is of Fano type) (cf.~\cite{kollar1998birational}*{Theorem 3.7} and \cite{zhang2010polarized}*{Lemma 2.11}).
		\item $X$ admits an int-amplified endomorphism (cf.~\cite{meng2020semi}*{Theorem 1.1}).
	\end{enumerate}
\end{lem}

\begin{lem}(cf.~\cite{meng2020building}*{Theorem 1.1})
	\label{lem-intamplified}
	Let $f$ be a surjective endomorphism of a projective variety $X$.
	Then the following are equivalent.
	\begin{enumerate}
		\item $f$ is int-amplified, i.e., $f^*L - L$ is ample for some ample Cartier divisor $L$ on $X$.
		\item All the eigenvalues of $f^*|_{\N^1(X)}$ are of modulus greater than $1$.
	\end{enumerate}
\end{lem}

\begin{lem}\label{lem-split-pdt}
	Let $f \colon X \to X$ be an int-amplified endomorphism on the product space $X \coloneqq X_1 \times X_2$ such that $f$ splits into $f_1 \times f_2$ after iteration (this is the case when both $X_i$ are of Fano type).
	Then every $f^{-1}$-periodic prime divisor $D$ is of the form $D_1 \times X_2$ or $X_1 \times D_2$ where $D_i$ is some prime divisor on $X_i$ for $i=1,2$.
\end{lem}

\begin{proof}
	Let $R_f$ and $R_{f_i}$ be the ramification divisors of $f$ and $f_i$, respectively.
	After iteration, we assume $f^*D = qD$ for some $q > 1$ (cf.~Lemma~\ref{lem-intamplified});
	thus $D$ is a component of $R_f$.
	Since $f$ splits, we have $R_f = p_1^* R_{f_1} + p_2^* R_{f_2}$ with $p_i$ being the natural projections.
	Hence, $D$ is of the form $D_1 \times X_2$ or $X_1 \times D_2$ where $D_i$ is some prime divisor on $X_i$.
\end{proof}

\begin{lem}(cf.~\cite{broustet2014singularities}*{Theorem 1.4} and \cite{meng2020nonisomorphic}*{Theorem 2.11})
	\label{thm-bh}
	Let $f \colon X \to X$ be an int-amplified endomorphism on a normal projective variety.
	Let $\Delta$ be an $f^{-1}$-invariant reduced divisor such that $K_X + \Delta$ is $\bQ$-Cartier.
	Then
	\begin{enumerate}
		\item $(X, \Delta)$ has at worst log canonical singularities, and
		\item $-(K_X + \Delta)$ is effective.
	\end{enumerate}
\end{lem}

The following Lemma~\ref{lem-blowup-or-conic} was first proved in \cite{ando1985extremal}*{Theorems 2.3 and 3.1} and then reformulated in \cite{wisniewski1991contractions}*{Theorem 1.2} (cf.~\cite{romano2019note}*{Theorem 1.2 and Remark 2.4}).

\begin{lem}\label{lem-blowup-or-conic}
	Let $\chi \colon X \to X'$ be a contraction of a $K_X$-negative extremal ray of a smooth projective variety $X$.
	If every fibre of $\chi$ is of dimension $\leq 1$, then $X'$ is smooth and either
	\begin{enumerate}
		\item $\chi$ is a conic bundle, or
		\item $\chi$ is a blow-up of $X'$ along a smooth subvariety of codimension $2$.
	\end{enumerate}
\end{lem}

\begin{lem}[cf.~\cite{casagrande2012picard}*{Theorem 1.1}]\label{lem-leftschetz-defect}
	Let $X$ be a smooth Fano fourfold and $\delta_X$ its Lefschetz defect.
	If $\delta_X \ge 4$, then $X \cong S_1 \times S_2$ with $S_i$ being del Pezzo surfaces.
	If $\delta_X = 3$, then there exists a conic bundle $X \to Y$ such that $\rho(X) - \rho(Y) = 3$.
\end{lem}
\begin{proof}
	The first assertion is a direct consequence from \cite{casagrande2012picard}*{Theorem 1.1}.
	For the second assertion, see \cite{casagrande2012picard}*{Step 3.3.15 of Proof of Proposition 3.3.1}.
\end{proof}

\begin{lem}[cf.~\cite{wisniewski1991contractions}*{pp. 156, Corollary}]
	\label{lem-base-Fano}
	Let $\tau \colon X \to Y$ be a conic bundle.
	If $X$ is smooth Fano and either $\dim Y \leq 3$ or $\rho(Y) = 1$, then $Y$ is also smooth Fano.
\end{lem}

Lemma \ref{lem-Rom4.2} below characterizes the behaviour of the non-elementary Fano conic bundles. 
We refer readers to Section \ref{section-main-reduction} for descriptions with a dynamical assumption.

\begin{lem}[cf.~\cite{romano2019non}*{Theorem 4.2}]
	\label{lem-Rom4.2}
	Let $X$ be a smooth Fano fourfold and $\tau \colon X \to Y$ a conic bundle which factors as $X \xrightarrow{\pi} W \xrightarrow{\tau_0} Y$ where $\pi$ is a composition of blow-ups along disjoint smooth projective surfaces and $\tau_0$ is an elementary conic bundle.
	\begin{enumerate}[label=(\arabic*),ref=2.9\,(\arabic*)]
		\item \label{lem-Rom4.2-geq-4}
		      If $\rho(X) - \rho(Y) \geq 4$, then $X \cong S_1 \times S_2$,
		      where $S_i$ are del Pezzo surfaces.
		\item \label{lem-Rom4.2-=3}
		      If $\rho(X) - \rho(Y) = 3$, then $\tau_0$ is a smooth $\bP^1$-bundle, and there exists a smooth $\bP^1$-bundle $p \colon Y \to Z$ with $Z$ being a del Pezzo surface.
		\item \label{lem-Rom4.2-=2}
		      If $\rho(X) - \rho(Y) = 2$, and $\tau_0$ is singular, then there exists a smooth $\bP^1$-bundle $p \colon Y \to Z$ with $Z$ being a del Pezzo surface.
		\item \label{lem-split-P^1-bundle}
		      If $X\cong S_1\times S_2$ with $S_i$ being del Pezzo surfaces (this is the case when (1) occurs), then $\tau_0$ is a splitting $\bP^1$-bundle.
	\end{enumerate}
\end{lem}

\begin{proof}
	All the assertions follow from \cite{romano2019non}*{Theorem 4.2} and note that (4) is a consequence of the proof of \cite{romano2019non}*{Theorem 4.2\,(1)}.
\end{proof}


The following Lemma~\ref{lem-blowup-basechange} is well known and we rewrite it here for readers' convenience.

\begin{lem}\label{lem-blowup-basechange}
	Let $W \to Z$ be a flat morphism of algebraic varieties.
	Let $C \subseteq Z$ be a closed subscheme and $D$ the inverse image of $C$ in $W$.
	Let $W'$ be the blow-up of $W$ along $D$ and $Z'$ the blow-up of $Z$ along $C$.
	Then $W' \cong W \times_Z Z'$.
\end{lem}

In what follows, we prove two lemmas that naturally complement Notation~\ref{notation-conic-bundle}.
Lemma~\ref{lem-no-3-comp} was first proved in \cite{sarkisov1983conic}*{Proposition 1.16}.
It shows, at each singular point of the discriminant of a conic bundle, there are precisely two components intersecting transversally.
We give an alternative proof here for readers' convenience.

\begin{lem}\label{lem-no-3-comp}
	Let $D_i \subseteq \Delta_{\tau}$ $(i = 1, 2, 3)$ be three distinct irreducible components (if exists) of a conic bundle $\tau \colon X \to Y$.
	Then $D_1 \cap D_2 \cap D_3 = \emptyset$.
\end{lem}

\begin{proof}
	Suppose the contrary that there exists $v \in D_1 \cap D_2 \cap D_3 \subseteq \Sing \Delta_{\tau}$.
	Choose an affine open neighbourhood $U \subseteq Y$ of $v$ with local coordinates $y_i$ such that one can write the local equation of $X_U \coloneqq \tau^{-1}(U) \subseteq U \times \bP^2_{x_0, x_1, x_2}$ in the form $Q_U \coloneqq \sum_{0 \leq i, j \leq 2} a_{ij} x_i x_j = 0$, where $x_i$ are the coordinates in $\bP^2$ and $a_{ij} \in \cO_Y(U)$.
	Locally, $\Delta_{\tau}$ is given by the determinant equation $\det Q_U = 0$.
	Since $\Delta_{\tau}$ has at least three components near $v$, we have $\det Q_U \in \fm_v^3$.
	It follows that $\rk (\partial^2 \det Q_U/\partial y_i \partial y_j)_v = 0$, contradicting \cite{sarkisov1983conic}*{Proposition 1.8 5.c}.
\end{proof}

\begin{lem}\label{lem-embedded-stable-under-BC}
	Let $\tau \colon X \to Y$ be a conic bundle, and $\pi \colon Y' \to Y$ a morphism from a smooth projective variety \(Y'\) such that $\pi(Y') \not\subset \Delta_{\tau}$.
	Then the morphism $\tau' \colon X' \to Y'$ by the base change is an embedded conic.
	In particular, if $Y'\subseteq Y$ is a smooth closed subvariety not contained in $\Delta_{\tau}$, and $\Delta_{\tau'} \coloneqq \Delta_{\tau}|_{Y'}$ is a reduced divisor with simple normal crossings on $Y'$, then $\tau'$ is also a conic bundle.
\end{lem}

\begin{proof}
	Since $\tau$ is a conic bundle, $\cE \coloneqq \tau_* \cO(-K_X)$ is locally free of rank $3$ on $Y$ (cf.~\cite{hartshorne1977algebraic}*{Chapter \RomanNumeralCaps{3}, Corollary 12.9}).
	Let $X' \coloneqq X \times_Y Y'$, $\widetilde{X} \coloneqq \bP_Y(\cE)$ and $\widetilde{X'} \coloneqq \bP_{Y'}(\pi^*\cE)$.
	\[
		\xymatrix@C=3.5pc{
		\widetilde{X'} \ar[r]^q \ar@/_2pc/[dd]_{\varphi'}	&\widetilde{X} \ar@/^2pc/[dd]^{\varphi} \\
		X' \ar[r]^p \ar@{^{(}->}[u]^{i'} \ar[d]_{\tau'}		&X \ar@{^{(}->}[u]_i \ar[d]^{\tau} \\
		Y' \ar[r]_{\pi}										&Y
		}
	\]

	Clearly, the generic fibre of $\tau'$ is an irreducible (and reduced) rational curve by the base change.
	Note that \(\tau'\) is flat (and proper) and hence \(X'\) is irreducible and reduced (cf.~e.g.~\cite{liu2002algebraic}*{Chapter~4, Proposition 3.8}).
	By \cite{sarkisov1983conic}*{\S 1.5}, $X$ embeds into $\widetilde{X}$ and $\tau = \varphi \circ i$.
	Then $\widetilde{X'} \cong \widetilde{X} \times_Y Y'$ with $\varphi'$ and $q$ being the natural projections.
	Since $\pi \circ \tau' = \varphi \circ i \circ p$, our $\tau'$ factors through $\varphi'$ by the universal property and we get the morphism $i'$.
	Then $i'$ is an embedding since $X' \cong X \times_{\widetilde{X}} \widetilde{X'}$;
	hence $\tau'$ is an embedded conic.
	The second part follows from the first part, $\tau$ being a (regular) conic bundle and \cite{sarkisov1983conic}*{Corollary 1.11}.
\end{proof}

At the end of this section, we recall the following lemma, which works in the proof of Lemma~\ref{lem-semistable-split} under dynamical assumptions.
Note that the system of walls given by \cite{greb2017compact}*{Theorem 6.6} yields an obvious stratification of $P(X)$ into connected chambers.
\begin{lem}(cf.~\cite{greb2017compact}*{Proposition 6.5 and Theorem 6.6})\label{lem-locally-finite-walls}
	There is a homeomorphism from $\Amp(Y)$ to $P(Y)$ (cf.~Notation~\ref{notation-walls}).
	The set of walls $\{W_{\cE}(\cF)\}_{\cF}$ is locally finite in $P(Y)$, i.e., there are only finitely many walls $W_{\cE}(\cF)$ in each compact set $K \subseteq P(Y)$.
\end{lem}

\section{Totally periodic subvarieties on Fano threefolds}

In this section, we shall study \cite{meng2020nonisomorphic}*{Question 1.8} for Fano threefolds.
The main results are Theorem~\ref{thm-TID-Fano-3fold} (confirming the divisor case) and Proposition~\ref{pro-ti-curves}.

\begin{thm}\label{thm-TID-Fano-3fold}
	Let $Y$ be a smooth Fano threefold admitting an int-amplified endomorphism $g$.
	Then $Y$ is toric and there is a toric pair $(Y, \Delta)$ such that $\Delta$ contains the union $\Sigma$ of all the $g^{-1}$-periodic prime divisors.
	Further, every $g^{-1}$-periodic prime divisor is a smooth rational surface.
\end{thm}


\begin{proof}
	By \cite{meng2020semi}*{Corollary 3.8}, there are only finitely many $g^{-1}$-periodic subvarieties;
	hence we may assume that they are all $g^{-1}$-invariant, after iteration.
	Further, $(Y, \Sigma)$ is lc and $-(K_Y + \Sigma)$ is effective (cf.~Lemma~\ref{thm-bh}).
	By Theorem~\ref{thm-threefold-class}, $Y$ is toric.
	We shall treat all cases of $Y$ in Theorem~\ref{thm-threefold-class} and replace $g$ by a power if necessary (cf.~Lemma \ref{lem-equiv-MMP}).

	\par \vskip 0.3pc \noindent
	\textbf{Case (1):} $Y \cong \bP^3$.
	Then $\Sigma$ is a union of at most four planes (cf.~\cite{horing2017totally}*{Corollary 1.2}, \cite{nakayama2010polarized}*{Theorem 1.5\,(5)}) by applying Lemma~\ref{thm-bh}~(2) to the pair \((Y,\Sigma)\).
	Since $(Y, \Sigma)$ is lc, the boundary $\Sigma$ is a simple normal crossing divisor (cf.~\cite{kollar1998birational}*{Lemma 2.29});
	thus $\Sigma$ is contained in a tetrahedron $\Delta$ (looking like \tikztetrahedron) in $\bP^3$.
	Clearly, $(Y, \Delta)$ is a toric pair.


	\par \vskip 0.3pc \noindent
	\textbf{Case (2):} $\pi \colon Y \to Z \cong \bP^3$ is a (toric) blow-up along an $(h \coloneqq g|_Z)^{-1}$-invariant line $L$ after iteration.
	Then, $\Sigma$ consists of the $\pi$-exceptional divisor $E$ and the $\pi$-strict transform of the $h^{-1}$-invariant divisor $\Sigma_Z \coloneqq \pi_*\Sigma$ on $Z$ (cf.~\cite{cascini2020polarized}*{Lemma 7.5}).

	\begin{claim}\label{claim-L-subseteq-intersection}
		There exists a reduced divisor $\Delta_Z$ containing $\Sigma_Z$ such that $(Z, \Delta_Z)$ is a toric pair and $L$ is contained in the intersection of toric boundary components of $(Z, \Delta_Z)$.
	\end{claim}

	\noindent
	\textbf{Proof of Claim~\ref{claim-L-subseteq-intersection}.}
	If $\Sigma_Z$ has at most one component, then we can choose a suitable $\Delta_Z$ satisfying the condition of our claim.
	Assume $D_1, D_2, \cdots$ are irreducible components of $\Sigma_Z$ and denote $\ell_{ij} \coloneqq D_i \cap D_j$.
	If the blown-up line $L = \ell_{ij}$ for some \(i,j\), then we are done (cf.~\textbf{Case (1)});
	hence we assume this is not the case.
	If $L \not\subset D_i$, $L\not\subset D_j$ but $L \cap \ell_{ij} \neq \emptyset$,
	then the three $g^{-1}$-invariant prime divisors $E, \pi^*D_i$ and $\pi^*D_j$ intersect along the curve $\pi^{-1}(L\cap\ell_{ij})$;
	consequently, applying Lemma~\ref{thm-bh}~(1) to the pair \((Y,E+\pi^*D_i+\pi^*D_j)\), we get a contradiction by noting that \((Y,E+\pi^*D_i+\pi^*D_j)\) is not log canonical (cf.~\cite{kollar1998birational}*{Lemma 2.29}).
	Hence, for each $\ell_{ij}$, either $L \cap \ell_{ij} = \emptyset$, or $L$ is contained in $D_i$ or $D_j$.

	If $\Sigma_Z$ consists of two components, then we can choose a suitable $\Delta_Z$ satisfying the condition.
	Note that every irreducible component of \(\Delta_Z\) is ample and there are at most two $(g|_L)^{-1}$-invariant points in $L$ by applying Lemma~\ref{thm-bh}~(2) to \(L\simeq \mathbb{P}^1\).

	Suppose that $\Sigma_Z$ has at least three components \(D_1,D_2\) and \(D_3\).
	Then $L$ must lie in one of them.
	Otherwise, by the previous argument, $L$ intersects them at three distinct points \(p,q,r\) away from \(l_{12},l_{13}\) and \(l_{23}\); in particular, \(p,q,r\) are $(g|_L)^{-1}$-invariant after iteration, and we deduce a contradiction by applying Lemma~\ref{thm-bh}~(2) to the pair \((L,p+q+r)\).
	Assume $L \subseteq D_1$.
	Then $L \cap \ell_{1i} \cap \ell_{1j} = \emptyset$ by applying Lemma~\ref{thm-bh}~(1) to the pair \((D_1\cong\mathbb{P}^2, L + \ell_{1i} + \ell_{1j})\), since \(L, \ell_{1i}, \ell_{1j}\) are $(g|_{D_1})^{-1}$-invariant.
	Hence if $\Sigma_Z$ has three components, we may choose another plane $D \supseteq L$ such that $\Delta_Z = \Sigma_Z + D$.
	Finally, if $\Sigma_Z$ has four components, then $L$ intersects $\ell_{1i}$ $(i = 2, 3, 4)$ at three distinct points \(p,q,r\), a contradiction to Lemma~\ref{thm-bh}~(2), noting that \(-(K_L+p+q+r)\) is not effective.

	\par \vskip 0.3pc \noindent
	\textbf{We come back to the proof of Case (2).}
	By Claim~\ref{claim-L-subseteq-intersection}, $(Y, \Delta \coloneqq \pi_*^{-1} \Delta_Z + E)$ is a toric pair.
	For $E \neq D \subseteq \Delta$, if $L \subseteq \pi(D)$, then $D = \pi_*^{-1}(\pi(D)) \cong \bP^2$ by \textbf{Case (1)};
	if $L \not\subset \pi(D)$, then $D = \pi^{-1}(\pi(D)) \cong \bF_1$.
	On the other hand, $Y$ admits another Fano contraction to $L' \cong \bP^1$ along which, $E$ dominates $L'$.
	So $E$ admits two (distinct) rulings and hence $E\cong\bF_0$.
	As a result, each component of $\Sigma_Z$ is smooth rational.

	\par \vskip 0.3pc \noindent
	\textbf{Case (3):} $\pi \colon Y \to Y'$ is a (toric) blow-up along disjoint curves $C_i$ which are intersections of $(g|_{Y'})^{-1}$-invariant prime divisors, and $Y' \to Z$ is a splitting $\bP^1$-bundle over a smooth toric surface.
	By \cite{meng2020nonisomorphic}*{Theorem 3.3}, there exists a reduced divisor $\Delta_{Y'}$ containing all the $(g'\coloneqq g|_{Y'})^{-1}$-periodic prime divisors such that $(Y', \Delta_{Y'})$ is a (smooth) toric pair.
	Let $\sum E_i$ be the sum of $\pi$-exceptional divisors.
	Then $(Y, \Delta_Y \coloneqq \sum E_i + \pi_*^{-1} \Delta_{Y'})$ is a toric pair.
	Note that $Y$ is a conic bundle over $Z$ (cf.~Theorem \ref{thm-threefold-class}).
	So every ($g^{-1}$-invariant) $\pi$-exceptional divisor $E_i$ is a $\bP^1$-bundle over $C_i \cong \bP^1$ (cf.~\cite{mori1983fano}*{Proposition 6.3}, \cite{cascini2020polarized}*{Lemma 7.5} and \cite{meng2020nonisomorphic}*{Corollary 3.4}).
	Hence every non-\(\pi\)-exceptional $g^{-1}$-invariant prime divisor is the smooth blow-up of a component of $\Delta_{Y'}$, and thus smooth rational.
\end{proof}

As an application of Theorem~\ref{thm-TID-Fano-3fold}, the following corollary slightly generalizes \cite{meng2020nonisomorphic}*{Theorem 3.3} to higher dimensional cases.
The proof is the same as \cite{meng2020nonisomorphic}*{Theorem 3.3} after replacing \cite{meng2020nonisomorphic}*{Theorem 3.2} by Theorem~\ref{thm-TID-Fano-3fold}.

\begin{cor}\label{cor-splitting-to-toric}
	Let $\tau \colon X \to Y$ be a splitting $\bP^1$-bundle over a smooth Fano threefold $Y$.
	Suppose that $X$ admits an int-amplified endomorphism $f$.
	Then $X$ is toric and there is a toric pair $(X, \Delta)$ such that $\Delta$ contains the union $\Sigma$ of all the $f^{-1}$-periodic prime divisors.
\end{cor}

At the end of this section, we show the following proposition, which takes a first glance at the totally periodic curves of an int-amplified endomorphism on Fano threefolds.
So far, we are only able to deal with the case when $Y$ admits a conic bundle.
It is conjectured that every totally periodic curve on $\bP^3$ is linear (cf.~e.g.,~\cite{meng2020nonisomorphic}*{Conjecture 1.9}).

\begin{prop}\label{pro-ti-curves}
	Let $Y$ be a smooth Fano threefold admitting a conic bundle and an int-amplified endomorphism $g$.
	Then every $g^{-1}$-periodic curve is smooth rational.
\end{prop}

\begin{proof}
	Let $C$ be a $g^{-1}$-periodic curve.
	Let $Y \to Z$ be the conic bundle, which factors as $Y \xrightarrow{\varphi} Y' \xrightarrow{p_0} Z$ such that $\varphi$ is the blow-up of $W$ along a disjoint union of some $(g|_{Y'})^{-1}$-periodic smooth curves, and $p_0$ is an algebraic $\bP^1$-bundle over a smooth rational surface $Z$ (cf.~\cite{meng2020nonisomorphic}*{Theorem 6.2}).
	By Lemma~\ref{lem-equiv-MMP}, we may assume that both $\varphi$ and $p_0$ are $g$-equivariant and $C$ is $g^{-1}$-invariant after iteration.

	Suppose that $C$ is contained in some ($g^{-1}$-invariant) $\varphi$-exceptional prime divisor $E$.
	Then $C$ is a $(g|_E)^{-1}$-invariant curve on the smooth rational surface $E$ (cf.~Theorem~\ref{thm-TID-Fano-3fold});
	thus $C$ is a smooth rational curve by \cite{meng2020nonisomorphic}*{Corollary 3.4}.

	Suppose that $C$ is not contained in the exceptional locus of $\varphi$.
	Then $C$ is birational to its image $C' \coloneqq \varphi(C)$ on $Y'$, which is $(g|_{Y'})^{-1}$-invariant (after iteration) by \cite{cascini2020polarized}*{Lemma 7.5}.
	If $p_0$ contracts $C'$, then $C'\cong\bP^1$.
	If $p_0$ does not contract $C'$, then $p_0(C')$ is $(g|_Z)^{-1}$-periodic and hence a smooth rational curve (cf.~\cite{cascini2020polarized}*{Lemma 7.5} and \cite{meng2020nonisomorphic}*{Corollary 3.4}).
	After iteration, $C'$ is contained in the $g^{-1}$-invariant smooth rational surface $F' \coloneqq p_0^{-1}(p_0(C'))$.
	Then $C'$, as a $(g|_{F'})^{-1}$-invariant curve on $F'$, is also smooth rational by \cite{meng2020nonisomorphic}*{Corollary 3.4}.
	In both cases, $\varphi|_C$ is an isomorphism.
	Indeed, taking the normalization $\widetilde{C} \to C \to C'$, the composition is birational and thus an isomorphism, which forces $\widetilde{C} \cong C$.
	So $C$ is smooth rational and our proposition holds.
\end{proof}

\section{Elementary Fano conic bundles}
\label{section-elementary-Fano-conic}

The whole section is devoted to proving the smoothness of an elementary Fano conic bundle with the dynamical assumption.
To be more precise, we shall prove the following:
\begin{thm}\label{thm-elementary-smoothness}
	Let $\tau \colon X \to Y$ be an elementary Fano conic bundle from a smooth fourfold $X$.
	Suppose that $X$ admits an int-amplified endomorphism $f$.
	Then $\Delta_{\tau} = \emptyset$, i.e., $\tau$ is a smooth (and hence an algebraic) $\bP^1$-bundle.
\end{thm}

\begin{notn}\label{notation-main-reduction-mmp}
	We will assume and use the following notation throughout this section.
	\begin{enumerate}
		\item $X$ is a smooth Fano fourfold, and $\tau \colon X \to Y$ is an elementary conic bundle;
		      hence $Y$ is a smooth Fano threefold (cf.~Lemma~\ref{lem-base-Fano}).
		\item $f \colon X \to X$ is int-amplified and it descends to an int-amplified endomorphism $g \coloneqq f|_Y$ on $Y$ after iteration (cf.~Lemmas~\ref{lem-equiv-MMP} and \ref{lem-intamplified}).
		\item We use $E_{\bullet}$ to denote the sum of all exceptional divisors  $\Exc(\bullet)$ for simplicity.
	\end{enumerate}
\end{notn}

Before proving Theorem~\ref{thm-elementary-smoothness}, we prepare two lemmas (cf.~Lemmas~\ref{lem-Z-to-minimal} and \ref{lem-Y-to-minimal}).

\begin{lem}[cf.~\cite{casagrande2008quasi}*{Theorem 3.14}]
	\label{lem-Z-to-minimal}
	Let $p \colon Y \to Z$ be a $\bP^1$-bundle over a smooth projective surface $Z$ and $\phi \colon Z \to Z'$ the blow-up of a point $Q$ on a smooth projective surface $Z'$.
	Then $Z'$ is Fano.
	Moreover, one has the following commutative diagram
	\[
		\xymatrix{
			Y \ar[r]^{\varphi} \ar[d]_{p} &Y' \ar[d]^{p'} \\
			Z \ar[r]_{\phi} &Z'
		}
	\]
	where $\varphi$ is the blow-up of a smooth Fano threefold $Y'$ along a smooth rational curve, the exceptional divisor $E_{\varphi} = p^*E_{\phi}$, and $p'$ is a Fano contraction and further a $\bP^1$-bundle.
\end{lem}

\begin{proof}
	First, $Z$ is Fano (cf.~Lemma~\ref{lem-base-Fano}) and thus $Z'$ is also Fano by the ramification divisor formula of $\phi$.
	Let $R_{\phi}$ (resp. $R_p$) be the extremal ray of $\NE(Z)$ (resp.~$\NE(Y)$) contracted by $\phi$ (resp. $p$).
	Given that $Y$ is Fano and thus $\NE(Y)$ is polyhedral, there is an extremal ray $R_{\varphi}$ of $\NE(Y)$ such that $p_* R_{\varphi} = R_{\phi}$ and $R_{\varphi} \cap R_p = \{0\}$.
	Denote by $\varphi \colon Y \to Y'$ the ($K_Y$-negative) contraction of $R_{\varphi}$.
	By the rigidity lemma (cf.~\cite{debarre2001higher}*{Lemma 1.15}), $\phi \circ p$ factors through $\varphi$ and we get the above commutative diagram.

	By \cite{casagrande2008quasi}*{Theorem 3.14\,(v)}, $\varphi$ is the blow-up of a smooth Fano threefold $Y'$ along a smooth rational curve with $E_{\varphi} = p^*E_{\phi}$.
	By Theorem~\ref{thm-TID-Fano-3fold}, $E_{\varphi}$ being $g^{-1}$-invariant is smooth rational.
	Since the fibre of $p'$ over $Q$ is $\varphi(E_{\varphi}) \cong \bP^1$ and every other fibre of $p'$ (which is isomorphic to the corresponding fibre of $p$) is $\bP^1$,
	our $p'$ is a $\bP^1$-bundle.
\end{proof}


\begin{lem}\label{lem-Y-to-minimal}
	Suppose $Y$ is imprimitive.
	Then one has the following commutative diagram
	\[
		\xymatrix{
			X \ar[r]^{\chi} \ar[d]_{\tau}	&X' \ar[d]^{\tau'} \\
			Y \ar[r]_{\varphi}				&Y'
		}
	\]
	such that the following assertions hold.
	\begin{enumerate}[label=(\arabic*),ref=4.4\,(\arabic*)]
		\item $\varphi$ is the blow-up of a smooth Fano threefold $Y'$ along a smooth curve $C$.
		\item \label{lem-Y-to-minimal-blow-up} $\chi$ is the blow-up of a smooth Fano fourfold $X'$ along a smooth projective surface.
		\item The exceptional divisor $E_{\chi} = \tau^*E_{\varphi}$.
		\item \label{lem-Y-to-minimal-conic}
		      $\tau'$ is an elementary Fano conic bundle and $X \cong X' \times_{Y'} Y$.
		      In particular, if $\tau'$ is a smooth $\bP^1$-bundle, then so is $\tau$.
		\item The above commutative diagram is $f$-equivariant after replacing $f$ by a power.
	\end{enumerate}
\end{lem}

\begin{proof}
	(1) follows from the imprimitivity of $Y$.
	Let $R_{\varphi}$ be the $K_Y$-negative extremal ray contracted by $\varphi$.
	Since $\NE(X)$ is rational polyhedral and $\tau$ is a $K_X$-negative contraction of an extremal ray $R_{\tau}$, there is an extremal ray $R_{\chi}$ of $\NE(X)$ such that $\tau_* R_{\chi} = R_{\varphi}$ and $R_{\chi} \cap R_{\tau} = \{0\}$, noting that faces of $\NE(Y)$ are in bijection with the faces of $\NE(X)$ containing $R_{\tau}$.
	Since $X$ is Fano, there exists a $K_X$-negative contraction $\chi \colon X \to X'$ of $R_{\chi}$.
	By the rigidity lemma, $\varphi\circ\tau$ factors through $\chi$.
	So (5) follows from Lemma~\ref{lem-equiv-MMP}.

	We claim that all the fibres of $\chi$ have dimension $\leq 1$.
	Indeed, if there exists a fibre component $F_0$ of $\chi$ with $\dim F_0 \geq 2$, then $\varphi(\tau(F_0))$ is a point.
	Since $R_{\chi} \cap R_{\tau} = \{0\}$, the restriction $\tau|_{F_0}$ is finite; thus $\tau(F_0)$ cannot be contracted to a point along $\varphi$ by (1).
	Hence our claim holds.
	Since $\tau_* R_{\chi} = R_{\varphi}$, we have $\tau(E_{\chi}) \subseteq E_{\varphi}$;
	then $\chi$ is birational with $E_{\chi} \subseteq \tau^* E_{\varphi}$.
	By Lemma~\ref{lem-blowup-or-conic}, $\chi$ is the blow-up along a smooth projective surface.
	Since $\tau^* E_{\varphi}$ is irreducible (cf.~\cite{kollar1998birational}*{Theorem 3.7}), (3) is proved.

	Next we show that $X'$ is Fano.
	Suppose the contrary.
	Then it follows from \cite{wisniewski1991contractions}*{Proposition 3.4} that there exists an extremal ray $R_1\neq R_{\chi}$ of $\NE(X)$ such that $E_{\chi} \cdot R_1 < 0$; thus the locus of $R_1$ is contained in $E_{\chi}$.
	By (3), $R_1$ is not contracted by $\tau$.
	Denote by $\chi_1 \colon X \to X_1$ the ($K_X$-negative) contraction of $R_1$.
	\textbf{We claim that all fibres of $\chi_1$ have dimension $\leq 1$.} Indeed, if there exists a fibre component $F_0$ of $\chi_1$ such that $\dim F_0 \geq 2$,
	then $E_{\chi} \cong \bP^1 \times \bP^1 \times \bP^1$ (cf.~\cite{wisniewski1991contractions}*{Proof of Proposition 3.6\,(ii)}).
	As a result, $\chi_1|_{E_{\chi}}$ is a Fano contraction onto $\bF_0$, a contradiction to $\dim F_0\ge 2$.
	So our claim holds.
	Then, it follows from Lemma~\ref{lem-blowup-or-conic} that $\chi_1$ is the blow-up of smooth $X_1$ along a smooth surface with $E_{\chi_1} = E_{\chi}$ and $E_{\chi} \to \chi_1(E_{\chi})$ is a smooth $\bP^1$-bundle.
	Since the ruling $\chi_1|_{E_{\chi}}$ gives rise to a ruling on $E_{\varphi}$ different from $\varphi|_{E_{\varphi}}$, we see that $E_{\varphi} \cong \bF_0$.
	Moreover, for a fibre $\ell_1$ of $\chi_1$, by the projection formula, we have $E_{\varphi} \cdot \tau(\ell_1) = E_{\chi} \cdot \ell_1 = -1$.
	Therefore, $E_{\varphi}|_{E_{\varphi}} \cong \cO_{E_{\varphi}}(-1,-1)$ and thus $Y'$ is not Fano (cf.~\cite{mori1983fano}*{Lemma 4.4}), a contradiction to (1).
	So our assumption is absurd and (2) is proved.

	Note that $\tau'$ is an elementary Fano contraction.
	For any point $y' \in Y' \setminus C$, we have $X'_{y'} \cong X_{\varphi^{-1}(y')}$ and hence $\dim X'_{y'} = 1$.
	On the other hand, $\chi(E_{\chi}) = \tau'^{-1}(C) \to C$ is flat (onto a smooth curve) and hence also has fibre dimension one.
	By Lemma~\ref{lem-blowup-or-conic}, $\tau'$ is an elementary (flat) conic bundle.
	So (4) follows from Lemma~\ref{lem-blowup-basechange}.
\end{proof}

\noindent
\textbf{Now we begin to prove Theorem~\ref{thm-elementary-smoothness}.}
In view of Lemmas~\ref{lem-Z-to-minimal}, \ref{lem-Y-to-minimal} and Theorem~\ref{thm-threefold-class}, to prove Theorem~\ref{thm-elementary-smoothness}, it suffices to focus on the following cases: (1) $Y \cong \bP^3$;
(2) $p \colon Y \to Z \cong \bP^2$ is a splitting $\bP^1$-bundle;
and (3) $p \colon Y \to Z \cong \bF_0$ is a splitting $\bP^1$-bundle.

Suppose the contrary that $\Delta_{\tau} \neq \emptyset$.
By \cite{cascini2020polarized}*{Lemma 7.4}, $g^{-1}(\Delta_{\tau}) = \Delta_{\tau}$.
After iteration, we may assume that each component $D_i$ of $\Delta_{\tau}$ is $g^{-1}$-invariant.
By Theorem~\ref{thm-TID-Fano-3fold}, $\Delta_{\tau}$ is contained in some toric boundary of $Y$ and thus $\Delta_{\tau}$ has simple normal crossings.

\begin{thm}\label{thm-smooth-P3}
	Suppose that $Y \cong \bP^3$.
	Then $\Delta_{\tau} = \emptyset$.
\end{thm}

\begin{proof}
	By \cite{horing2017totally}*{Corollary 1.2}, each component $D_i \cong \bP^2$.
	By Lemma~\ref{lem-no-3-comp} and the ampleness of each $D_i$, we see that $\Delta_{\tau}$ has at most two components.
	However, both $D_i \cong \bP^2$ and $D_i \setminus (D_i \cap D_j) \cong \bP^2 \setminus \bP^1$ are simply connected;
	thus there is no nontrivial {\'e}tale cover over them.
	This contradicts our $\tau$ being elementary (cf.~Notation~\ref{notation-conic-bundle}).
\end{proof}

\begin{thm}\label{thm-smooth-P1-bundle-over-P2}
	Suppose that $p \colon Y \to Z \cong \bP^2$ is a splitting $\bP^1$-bundle.
	Then $\Delta_{\tau} = \emptyset$.
\end{thm}

\begin{proof}
	Up to a twist, we may write $Y = \bP_Z(\cF)$ where $\cF \cong \cO_Z \oplus \cO_Z(a)$ with $-2 \leq a \leq 0$ by an easy calculation (cf.~\cite{szurek1990fano}).
	There are two possibilities:

	\par \vskip 0.3pc \noindent
	\textbf{Case 1: $\Delta_{\tau}$ contains no section of $p$.} Then each component $D_i$ of $\Delta_{\tau}$, being the pullback of some line on $Z \cong \bP^2$, is a $g^{-1}$-invariant Hirzebruch surface by Theorem~\ref{thm-TID-Fano-3fold}.
	Since $D_i \cong \bF_d$ and $D_i \setminus (D_i \cap D_j) \cong \bA^1 \times \bP^1$ are both simply connected,
	the double cover $\sigma \colon \widetilde{\Delta_{\tau}} \to \Delta_{\tau}$ being nontrivial implies that
	$\Delta_{\tau}$ is the pullback of a union of three $(g|_Z)^{-1}$-periodic lines with no common intersection (looking like \tikztriangle) on $Z \cong \bP^2$ (cf.~Lemma~\ref{thm-bh}).

	If $p$ is a trivial bundle with $Y \cong Z \times T \cong Z \times \bP^1$, then after iteration, $g = g|_Z \times g|_T$ with $g|_T$ being polarized on $T$ (cf.~Lemma~\ref{lem-intamplified}).
	So we can pick a $g|_T$-periodic point $t \in T$ (cf.~\cite{fakhruddin2003questions}*{Theorem 5.1}) and define $S \coloneqq Z \times \{t\} \cong \bP^2$, a $g$-periodic section of $p$.
	If $p$ is nontrivial, then we take $S$ to be the ``negative section'' of this splitting $\bP^1$-bundle such that $\cO_Y(S)|_S = \cO_S(a)$, which is $g^{-1}$-periodic (cf.~\cite{meng2020nonisomorphic}*{Lemma 2.3}).
	In both cases, $g|_S$ is int-amplified after iteration.
	Then the base change $\tau_S \colon X_S \coloneqq X \times_Y S \to S$ is proper and flat.
	Moreover, $S \not\subset \Delta_{\tau}$ and $\Delta_{\tau_S} = \Delta_{\tau}|_S$ is a loop consisting of three rational curves on $S$ with simple normal crossings.
	By Lemma~\ref{lem-embedded-stable-under-BC}, $X_S$ is smooth and $\tau_S$ is a conic bundle.
	Since $S \cong \bP^2$ and $\Delta_{\tau_S} \neq \emptyset$, one has $\rho(X_S/S) = 1$ (cf.~\cite{mori1983fano}*{Corollary 6.4}).
	On the other hand, $X_S$ is Fano.
	Indeed, for any curve $C \subseteq X_S$, we have
	\[
		K_{X_S} \cdot C = (K_X + X_S)|_{X_S} \cdot C = K_X \cdot C + X_S \cdot C = K_X \cdot C + (S \cdot \tau_* C)_Y < 0,
	\]
	since $\cO_{Y}(S)|_S \cong \cO(a)$ with $a \leq 0$.
	Since $\NE(X_S)$ has only two extremal rays, our $-K_{X_S}$ is ample.
	So we apply \cite{meng2020nonisomorphic}*{Theorem 4.1} to conclude $\Delta_{\tau_S} = \emptyset$, a contradiction.

	\par \vskip 0.3pc \noindent
	\textbf{Case 2: $\Delta_{\tau}$ contains at least one section $S$ of $p$.}
	Then $S \cong \bP^2$ and $S \setminus (S \cap D_i) \cong \bP^2 \setminus \bP^1$ are both simply connected where $D_i$ is a vertical component (if exists).
	So there exist at least two components $D_1, D_2$ of $\Delta_{\tau}$, which are the pullback of some lines $\ell_1, \ell_2$ on $Z \cong \bP^2$.
	Given that $\ell_1 \cap \ell_2$ is nonempty in $\bP^2$, so is $S \cap D_1 \cap D_2$.
	But this violates Lemma~\ref{lem-no-3-comp}.

	So we finish the proof of our theorem.
\end{proof}

\begin{thm}\label{thm-smooth-P1-bundle-over-P1xP1}
	Suppose that $p \colon Y \to Z \cong \bF_0$ is a splitting $\bP^1$-bundle.
	Then $\Delta_{\tau} = \emptyset$.
\end{thm}

\begin{proof}
	According to \cite{szurek1990fano}, up to a twist, we may write $Y = \bP_Z(\cF)$ where $\cF \cong \cO_Z \oplus \cO_Z(a, b)$ with $(a, b) = (0, 0), (0, -1), (-1, -1)$ or $\cF \cong \cO_Z \oplus \cO_Z(1, -1)$.

	\par \vskip 0.3pc \noindent
	\textbf{Case 1: $\Delta_{\tau}$ contains no section of $p$.}
	Then, similar to the proof of Theorem~\ref{thm-smooth-P1-bundle-over-P2}, $\Delta_{\tau}$ is the pullback of a loop (looking like \tikzsharp) on $Z$.
	If $\cF \cong \cO_Z \oplus \cO_Z(1, -1)$, then for a general fibre $\ell$ of the first projection $Z \to \bP^1$, one has $\cF|_{\ell} \cong \cO_{\ell} \oplus \cO_{\ell}(-1)$.
	Hence $H \coloneqq p^* \ell \cong \bP_{\ell}(\cF|_{\ell})\cong\bF_1$, and $\tau_H \colon \tau^* H \to H$ is a Fano conic bundle with $\Delta_{\tau_H} = \Delta_{\tau}|_H$ being two fibres (cf.~Lemma~\ref{lem-embedded-stable-under-BC}), a contradiction to \cite{mori1983fano}*{Corollary 6.7}.
	In the remaining cases, $S|_S$ is anti-nef, and then with the same proof as in Theorem~\ref{thm-smooth-P1-bundle-over-P2}, we can deduce a contradiction to \cite{meng2020nonisomorphic}*{Theorem 4.1}.

	\par \vskip 0.3pc \noindent
	\textbf{Case 2: $\Delta_{\tau}$ contains at least one section $S\cong\bF_0$ of $p$.}
	After iteration, we may assume that $Y \to Z \coloneqq Z_1 \times Z_2 \to Z_i \cong \bP^1$ is $g$-equivariant (cf.~Lemma~\ref{lem-equiv-MMP}).
	Take a general point $z \in Z$ and let $Y_z \coloneqq p^{-1}(z)$ be its (movable) fibre.
	By Lemma~\ref{thm-bh}, we have $\Delta_{\tau} \cdot Y_z \leq -K_Y \cdot Y_z = 2$;
	thus $\Delta_{\tau}$ contains at most two sections of $p$.

	Since $S$ is simply connected, there exists a component $D_1 = p^* \ell_1 \subseteq \Delta_{\tau}$, with some curve $\ell_1$ on $Z$ (cf.~Notation~\ref{notation-elementary}).
	By \cite{cascini2020polarized}*{Lemma 7.5}, $\ell_1$ is $(g|_Z)^{-1}$-periodic.
	So we may assume $\ell_1$ is $\{\cdot\} \times \bP^1$ (cf.~\cite{meng2020nonisomorphic}*{Lemma 3.1}).
	Clearly, $D_1\cong\bF_d$ for some $d\ge 0$, and both $S \setminus (S \cap D_1) \cong \bA^1 \times \bP^1$ and $D_1 \setminus (S \cap D_1) \cong \bF_d \setminus (\text{a section of } \bF_d)$ are simply connected.
	On the other hand, by Lemma~\ref{lem-no-3-comp}, any three components of $\Delta_{\tau}$ have no common intersection.
	So the double cover $\sigma \colon \widetilde{\Delta_{\tau}} \to \Delta_{\tau}$ being nontrivial implies that $\Delta_{\tau}$ consists of exact two sections of $p$, and two disjoint components $D_1, D_2$ which are pullbacks of some curves $\ell_1, \ell_2$ on $Z$, respectively.
	In particular, $\ell_2$ is of the form $\{\cdot\} \times \bP^1$.

	Take a $g|_{Z_2}$-periodic point $\{z_2\} \in Z_2$ (cf.~Lemma~\ref{lem-intamplified} and \cite{fakhruddin2003questions}*{Theorem 5.1}) and define $L_{z_2} \coloneqq Z_1 \times \{z_2\}$.
	Then $Y_{z_2} \coloneqq p^* L_{z_2}$ is a Hirzebruch surface and $X_{z_2} \coloneqq (p \circ \tau)^* L_{z_2}$ is a Fano threefold admitting an int-amplified endomorphism after iteration.
	Since $Y_{z_2} \not\subset \Delta_{\tau}$ and $\Delta_{\tau_{z_2}} = \Delta_{\tau}|_{Y_{z_2}}$ is a loop of four rational curves (looking like \tikzsharp) on $Y_{z_2}$ with simple normal crossings, by Lemma~\ref{lem-embedded-stable-under-BC}, $X_{z_2}$ is smooth and $\tau_{z_2} \colon X_{z_2} \to Y_{z_2}$ is a conic bundle.
	Since $\Delta_{\tau_{z_2}}$ is connected and ample on $Y_{z_2}$, we have $\rho(X_{z_2}/Y_{z_2}) = 1$ (cf.~\cite{mori1983fano}*{Corollary 6.4}).
	By \cite{meng2020nonisomorphic}*{Theorem 4.1}, $\Delta_{\tau_{z_2}} = \emptyset$, a contradiction.
\end{proof}

\begin{proof}[\textbf{End of Proof of Theorem~\ref{thm-elementary-smoothness}.}]
	Theorem~\ref{thm-elementary-smoothness} follows from Theorems~\ref{thm-smooth-P3} $\sim$ \ref{thm-smooth-P1-bundle-over-P1xP1} and \cite{meng2020nonisomorphic}*{Lemma 2.12}, noting that $Y$ is rational (cf.~\cite{zhang2012rationality}*{Theorem 1.2}).
\end{proof}

\section{Main reduction for non-elementary conic bundles}
\label{section-main-reduction}

In this section, we shall study non-elementary (singular) conic bundles.
The main results are Theorem~\ref{thm-structure-mmp} and Lemma~\ref{lem-key-reduction}; see \cite{meng2020nonisomorphic}*{Theorem 6.2} for the threefold case.

\begin{thm}[Equivariant minimal model for conic bundles]
	\label{thm-structure-mmp}
	Let $X$ be a smooth Fano fourfold and $\tau \colon X \to Y$ a conic bundle.
	Suppose $X$ admits an int-amplified endomorphism $f$.
	Then, after iteration, there exists an $f$-equivariant minimal model program
	\[
		\xymatrix{
		X = X_0 \ar[r]^-{\pi_1} &X_1 \ar[r]^{\pi_2} &X_2 \ar[r]^{\pi_2} &{\cdots} \ar[r]^-{\pi_r} &X_r \eqqcolon W \ar[r]^-{\tau_0} &Y
		}
	\]
	such that the following assertions hold.
	\begin{enumerate}[label=(\arabic*),ref=5.1\,(\arabic*)]
		\item $r = \rho(X) - \rho(Y) - 1$ and each $X_i$ is a smooth projective fourfold.
		\item \label{thm-structure-mmp-algebraic}
		      $\tau_0 \colon W = \bP_Y(\cE) \to Y$ is an algebraic $\bP^1$-bundle over a smooth Fano threefold $Y$.
		\item There are $r$ disjoint prime divisors $D_1, \cdots, D_r$ on $Y$ and $r$ pairs of prime divisors $E_i, \widetilde{E_i}$ on $X$ such that $\tau^* D_i = E_i + \widetilde{E_i}$ and $D_i \cap D_j = \emptyset$ for $i \neq j$.
		\item $\tau_0$ is smooth and $\Delta_{\tau} = \bigsqcup_{i=1}^{r} D_i$ with each $D_i$ a $(g:=f|_Y)^{-1}$-invariant smooth rational connected component of $\Delta_{\tau}$.
		      Moreover, $\tau$ has reduced fibres over $\bigcup_{i=1}^rD_i$.
		\item The composition $X_i \to Y$ is a conic bundle with the discriminant $D_{i+1} \cup \cdots \cup D_r$.
		\item The composition $\pi=\pi_r\circ\cdots\circ\pi_1 \colon X \to W$ is the blow-up of $W$ along $r$ disjoint union of $(f|_W)^{-1}$-invariant smooth rational surfaces $\bigcup_{i=1}^r \overline{D_i}$ with $\tau_0|_{\overline{D_i}} \colon \overline{D_i} \cong D_i$.
	\end{enumerate}
\end{thm}

\begin{proof}
	First, by Lemma~\ref{lem-base-Fano}, $Y$ is smooth Fano.
	By \cite{romano2019non}*{Propositions 3.4, 3.5}, we may run a relative minimal model program $X \to X_1 \to \cdots \to X_i \to \cdots \to X_r$ of $X$ over $Y$ which is $f$-equivariant after iteration (cf.~Lemma~\ref{lem-equiv-MMP}), such that (1) and (3) hold.

	We will show the smoothness of $\tau_0$ in the following two paragraphs. 
	By Lemma~\ref{lem-Rom4.2} and Theorem~\ref{thm-elementary-smoothness}, we only need to consider the case $r = 1$, i.e., $\pi = \pi_1$ is a single blow-up along a smooth projective surface.
	Suppose the contrary that $\tau_0$ is singular, i.e., $\Delta_{\tau_0} \neq \emptyset$.
	By Lemma~\ref{lem-Rom4.2-=2}, there exists a (smooth) $\bP^1$-bundle $p \colon Y \to Z$ to a smooth rational surface.
	After iteration, we may assume $f$ descends to an int-amplified endomorphism $f|_Z$ on $Z$ and $\Delta_{\tau}$ is $g^{-1}$-invariant on $Y$ (cf.~Lemmas~\ref{lem-equiv-MMP}, \ref{lem-intamplified} and \cite{cascini2020polarized}*{Lemma 7.4}).
	By \cite{meng2020nonisomorphic}*{Theorem 3.3}, each component of $\Delta_{\tau}$ is either a section of $p$ or the pullback of some $(f|_Z)^{-1}$-periodic (rational) curve on $Z$ (cf.~\cite{cascini2020polarized}*{Lemma 7.5}).

	Suppose that $\Delta_{\tau_0}$ contains a section $S$ of $p$.
	Since $S \cong Z$ is simply connected, the nontrivial double cover $\widetilde{\Delta_{\tau_0}} \to \Delta_{\tau_0}$ (cf.~Notation~\ref{notation-conic-bundle}) implies that there is another irreducible (vertical) component $F$ of $\Delta_{\tau_0}$ intersecting $S$.
	Then, $F$ is the pullback of some $(f|_Z)^{-1}$-periodic (rational) curve on $Z$.
	But now there is no $D_1$ disjoint from $\Delta_{\tau_0}$, a contradiction.
	Hence each component of $\Delta_{\tau}$ is a pullback of some curve $C_i$ on $Z$.
	Since $C_i$'s are $(f|_Z)^{-1}$-periodic, there is a toric pair $(Z, \Delta_Z)$ such that each $C_i \subseteq \Delta_Z$ (cf.~\cite{meng2020nonisomorphic}*{Theorem 3.2}).
	Clearly, $\Delta_Z$ is a simple loop of smooth rational curves.
	Since $D_1$ is disjoint from the connected $\Delta_{\tau_0}$, there is some $C_i$ such that $p^{-1}(C_i) \subseteq \Delta_{\tau_0}$ and $C_i$ intersects exact one of other $C_j$'s.
	However, $p^{-1}(C_i) \setminus (p^{-1}(C_i \cap C_j)) \cong \bA^1 \times \bP^1$ being simply connected contradicts the existence of the nontrivial double cover (cf.~Notation \ref{notation-conic-bundle}).
	Therefore, $\tau_0$ is a smooth $\bP^1$-bundle.

	Since $\tau_0$ is smooth, by \cite{meng2020nonisomorphic}*{Lemma 2.12}, $W = \bP_Y(\cE)$ for some locally free sheaf $\cE$ of rank $2$ over $Y$, since $Y$ is rational (cf.~\cite{zhang2012rationality}*{Theorem 1.2}).
	So (2) is proved.

	Let $f_i \coloneqq f|_{X_i}$.
	Since the exceptional divisor $\Exc(\pi_i)$ is $f_{i-1}^{-1}$-invariant and the surface $D_i' \subseteq X_i$ blown up by $\pi_i$ is $f_i^{-1}$-invariant,
	its image $\overline{D_i}$ on $W$ is $f_r^{-1}$-invariant, and its image $D_i$ on $Y$ is $g^{-1}$-periodic (cf.~\cite{cascini2020polarized}*{Lemma 7.5}).
	In particular, Theorem~\ref{thm-TID-Fano-3fold} implies that $D_i$ is rational for all $i$.
	Together with (3), (4) $\sim$ (6) are proved.
\end{proof}

From now on till the end of this section, we assume Notation~\ref{notation-main-reduction-mmp} except that our $\tau$ here may not be elementary.
We shall prove Lemma~\ref{lem-key-reduction} below, which generalizes Lemma~\ref{lem-Y-to-minimal}.

\begin{lem}\label{lem-key-reduction}
	Let $\tau \colon X \to Y$ be a Fano conic bundle, which factors as $X \xrightarrow{\pi} W \xrightarrow{\tau_0} Y$ where $\pi$ is the blow-up along disjoint surfaces and $W = \bP_Y(\cE) \to Y$ is an algebraic $\bP^1$-bundle.
	Suppose that $Y$ is imprimitive.
	Then we have the following commutative diagram
	\[
		\xymatrix@C=5pc{
		X \ar[r]^{\chi} \ar@/_2pc/[dd]_{\tau} \ar[d]^{\pi} \ar@[blue][dr]^{\color{blue} \eta}	&X' \ar@[blue][d]_{\color{blue} \pi'} \ar@/^2pc/[dd]^{\tau'} \\
		W \ar[d]^{\tau_0} \ar@[blue][r]_{\color{blue} \psi}										&{\color{blue} W'} \ar@[blue][d]_{\color{blue} \tau_0'} \\
		Y \ar[r]_{\varphi}																		&Y'
		}
	\]
	such that the following assertions hold.
	\begin{enumerate}
		\item $\varphi$ is the blow-up of a smooth Fano threefold $Y'$ along a smooth curve $C$.
		\item $\chi$ is the blow-up of a smooth fourfold $X'$ along a smooth projective surface.
		\item The above commutative diagram is $f$-equivariant after replacing $f$ by a power.
	\end{enumerate}
	Moreover, if the exceptional divisor $E_{\varphi} \not\subset \Supp \Delta_{\tau}$, then the following assertions hold.
	\begin{enumerate}[label=(\roman*),ref=5.2\,(\roman*)]
		\item $\psi$ is a $K_W$-negative contraction and is a blow-up along a smooth projective surface.
		\item $E_{\psi} = \tau_0^*E_{\varphi}$ and $E_{\chi} = \pi^*E_{\psi} = \tau^*E_{\varphi}$.
		\item \label{lem-key-reduction-W'}
		      $\tau_0': W' = \bP_{Y'}(\cE') \to Y'$ is an algebraic $\bP^1$-bundle such that $\cE = \varphi^*\cE'$.
		\item $X'$ is Fano, and $\tau'$ is a Fano conic bundle.
		\item $W \cong W'\times_{Y'} Y$ and $X \cong X' \times_{Y'} Y \cong X'\times_{W'}W$.
	\end{enumerate}
\end{lem}


\begin{proof}
	(1) follows from the imprimitivity of $Y$.
	Let $R_{\varphi}$ be the $K_Y$-negative extremal ray contracted by $\varphi$.
	Since $\tau$ (resp.~$\tau_0$) is a $K_X$ (resp.~$K_W$)-negative contraction of an extremal face $F_{\tau}$ of $\NE(X)$ (resp.~ $F_{\tau_0}$ of $\NE(W)$), there are extremal rays $R_{\psi}$ and $R_{\chi}$ of $\NE(W)$ and $\NE(X)$, respectively such that $\tau_* R_{\chi} = (\tau_0)_* R_{\psi} = R_{\varphi}$ and $R_{\chi} \cap F_{\tau} = R_{\psi} \cap F_{\tau_0} = \{0\}$ (cf. Proof of Lemma \ref{lem-Y-to-minimal}).
	Since $X$ is Fano, there exists a $K_X$-negative contraction $\chi \colon X \to X'$ of $R_{\chi}$.
	By the rigidity lemma (cf.~\cite{debarre2001higher}*{Lemma 1.15}), 
	$\varphi\circ\tau$ factors through $\chi$.
	Then, the same proof of Lemma \ref{lem-Y-to-minimal-blow-up} shows (2) (here, $X'$ may not be Fano).

	Similarly, there exists a contraction $\eta \colon X \to W'$ of the $K_X$-negative extremal face $F_{\eta}$ containing the $\pi$-contracted extremal face $F_{\pi}$ such that $\pi_* F_{\eta} = R_{\chi}$.
	By the rigidity lemma several times, we see that $\eta$ factors through $\pi$, $\eta$ factors though $\chi$, and $\tau'$ factors through $\pi'$.
	So we get the commutative diagram and (3) follows from Lemma~\ref{lem-equiv-MMP}.

	\textbf{From now on, we further assume that $E_{\varphi} \not\subset \Supp \Delta_{\tau}$.}
	Then $\tau^* E_{\varphi}$ is irreducible.
	Since $\tau_* E_{\chi} \subseteq E_{\varphi}$, it follows that $\tau^* E_{\varphi} = E_{\chi}$ (cf.~\cite{kollar1998birational}*{Theorem 3.7}).
	With the same proof of Lemma \ref{lem-Y-to-minimal-blow-up}, our $X'$ is Fano.

	Note that different components of $E_{\pi}$ are disjoint from each other;
	hence $E_{\pi} \cdot \ell_{\pi} < 0$ for every fibre $\ell_{\pi}$ of $\pi$.
	Since $E_{\chi} \cdot \ell_\pi = 0$ by the projection formula, our \textbf{$E_{\chi}$ is not a component of $E_{\pi}$}.
	Let $\ell_{\psi} \in R_{\psi}$ be a curve on $W$, and $\ell_{\chi} \in R_{\chi}$ a curve on $X$.
	Then $\pi_* \ell_{\chi} = a \ell_{\psi}$ for some $a \in \bZ_{> 0}$, and we have
	\[
		K_W \cdot (a \ell_{\psi}) = \pi^* K_W \cdot \ell_{\chi} = (K_X - E_{\pi}) \cdot \ell_{\chi} < 0.
	\]
	Here, $E_{\pi} \cdot \ell_{\chi} \geq 0$;
	otherwise, $E_{\chi}$ will coincide with some component of $E_{\pi}$ by (2), a contradiction.
	So $\psi$ is a $K_W$-negative contraction.
	Applying \cite{casagrande2008quasi}*{Theorem 3.14} for the diagram $\tau_0' \circ \psi = \varphi \circ \tau_0$, (i) and (ii) are proved.

	With the same proof of Lemma~\ref{lem-Y-to-minimal-conic}, $\tau_0'$ is an elementary conic bundle.
	Note that, outside the curve $C \subseteq Y'$ blown up by $\varphi$, every fibre of $\tau_0'$ is a smooth conic since so is $\tau_0$.
	Thus the divisor $\Delta_{\tau_0'}\subseteq C$, which is absurd (cf.~Notation~\ref{notation-discriminant}).
	Hence, $\tau_0'$ is a smooth $\bP^1$-bundle.
	By (3) and \cite{zhang2012rationality}*{Theorem 1.2}, $Y'$ is rational;
	thus $W' = \bP_{Y'}(\cE')$ for some locally free sheaf $\cE'$ of rank $2$ on $Y'$ (cf.~\cite{meng2020nonisomorphic}*{Lemma 2.12}).
	By Lemma~\ref{lem-blowup-basechange}, $W \cong W' \times_{Y'} Y$, hence up to a twist, $\cE = \varphi^* \cE'$, which implies (iii).

	\textbf{We claim that $\tau'$ is equidimensional.}
	Consider the behaviour of $\pi'$ and note that $E_{\pi'} \subseteq \chi(E_{\pi})$ is a disjoint union of prime divisors.
	On the one hand, outside $\psi(E_{\psi})$, the fibres of $\pi'$ have dimension $\leq 1$ since so are fibres of $\pi$ over $W \setminus E_{\psi}$.
	On the other hand, $(\pi')^{-1}(\psi(E_{\psi})) = \chi(E_{\chi})$ is a smooth surface by (1);
	thus the fibres of $\pi'$ over $\psi(E_{\psi})$ have dimension $\leq 1$.
	So, all fibres of $\pi'$ have dimension $\leq 1$ and it follows from $X'$ being Fano, Lemma~\ref{lem-blowup-or-conic} and the induction on $\rho(X'/W')$ that $\pi'$ is the blow-up along disjoint surfaces $S'$ on $W'$.
	Moreover, $S'$ is the image under $\psi$ of the surfaces blown up by $\pi$ by the above diagram;
	hence $S'$ is the union of subsections of $\tau_0'$.
	This further implies the fibres of $\tau'$ over $C$ are of dimension $1$ and our claim holds.

	Since $X'$ is Cohen-Macaulay and $Y'$ is smooth, our $\tau'$ is flat.
	By (iii) and Lemma~\ref{lem-blowup-basechange}, (iv) and (v) are proved.
	So we complete the proof of Lemma~\ref{lem-key-reduction}.
\end{proof}


\begin{rem}\label{rem-reduction-lem}
	Removing the condition ``$E_{\varphi} \not\subset \Supp\Delta_{\tau}$'', we still have the commutative diagram in Lemma~\ref{lem-key-reduction}.
	However, in this case, $\psi$ is possibly not a $K_W$-negative contraction;
	hence Lemma~\ref{lem-blowup-or-conic} cannot be applied and $W'$ may not even be $\bQ$-factorial!
\end{rem}

\section{Conic bundles onto \texorpdfstring{$\bP^1$}{P1}-bundles over rational surfaces}
\label{section-splitting-ness-to-surface}

In this section, we consider the conic bundles with the base isomorphic to a $\bP^1$-bundle over a rational surface.
Theorem~\ref{thm-splitting-P1-bundle-over-minimal} is our main result in this section.


\begin{notn}\label{notation-splitting}
	We follow the notations below throughout this section.
	\begin{enumerate}[label=(\arabic*),ref=6.1\,(\arabic*)]
		\item\label{notation-splitting-factor} $\tau \colon X \to Y$ is a conic bundle from a smooth Fano fourfold $X$, which factors as $X \xrightarrow{\pi} W \xrightarrow{\tau_0} Y$ where $\pi$ is a composition of blow-ups along disjoint smooth projective surfaces and $\tau_0$ is an elementary conic bundle.
		\item By Theorem~\ref{thm-structure-mmp}, $Y$ is a smooth Fano threefold and $W = \bP_Y(\cE)$ with $\cE$ being a locally free sheaf of rank $2$ on $Y$.
		\item $f \colon X \to X$ is an int-amplified endomorphism.
		      After iteration, $f$ descends to $g \coloneqq f|_Y$ and $h \coloneqq f|_W$, which are both int-amplified (cf.~Lemmas~\ref{lem-equiv-MMP} and \ref{lem-intamplified}).
		\item $p \colon Y \to Z$ is a $\bP^1$-bundle over a smooth rational surface $Z$.
		      By \cite{meng2020nonisomorphic}*{Theorem 6.4}, $Y = \bP_Z(\cF)$ with $\cF$ being a splitting locally free sheaf of rank $2$ on $Z$.
		      \[
			      \xymatrix@C=2.5pc{
			      {f \acts X} \ar[r]_-{\pi} \ar@/^0.8pc/[rr]^-{\tau}
			      &{h \acts W} \ar[r]_-{\tau_0}
			      &{g \acts Y} \ar[r]_-{p}
			      &Z
			      }
		      \]
		\item For each $z \in Z$, let $Y_z \coloneqq p^{-1}(z)$, $X_z \coloneqq (p \circ \tau)^{-1}(z)$ and $W_z \coloneqq (p \circ \tau_0)^{-1}(z)$.
		\item We use $E_{\bullet}$ to denote the sum of all exceptional divisors $\Exc(\bullet)$ for simplicity.
	\end{enumerate}
\end{notn}

\begin{thm}\label{thm-splitting-P1-bundle-over-minimal}
	In the setting of Notation~\ref{notation-splitting}, suppose that $Z$ is isomorphic to $\bP^2$, $\bF_0$ or $\bF_1$.
	Then $\tau_0$ is a splitting $\bP^1$-bundle.
\end{thm}

In what follows, we generalize \cite{meng2020nonisomorphic}*{Lemma 6.3} to the following higher dimensional case.
The proof of Lemma~\ref{lem-fibre-fano} will last till the paragraph before Lemma \ref{lem-semistable-split}.
\begin{lem}\label{lem-fibre-fano}
	$p \circ \tau_0 \colon W \to Z$ is a fibre bundle such that fibres are either all $\bF_0$ or all $\bF_1$.
\end{lem}

\begin{proof}
	First, note that each fibre $W_z\cong \bF_d$ for some $d \geq 0$.
	So the lemma is equivalent to showing that $d \leq 1$, i.e., $W_z$ is Fano, since $\bF_0$ and $\bF_1$ cannot deform to each other.

	Let $r \coloneqq \rho(X) - \rho(W)$.
	If $r = 0$, then $W=X$ is Fano;
	hence our lemma follows from the adjunction formula.
	So we may assume $r > 0$.
	By Theorem~\ref{thm-structure-mmp}, $\Delta_{\tau}$ is a disjoint union of $r$ smooth $g^{-1}$-invariant surfaces $D_i$, and $\pi$ is the blow-up of $W$ along the disjoint union $\bigsqcup_{i=1}^{r} \overline{D_i}$ of 
	$h^{-1}$-invariant surfaces $\overline{D_i}$.

	Write $\tau^{-1}(D_i) = \tau^* D_i = E_i + \widetilde{E_i}$ where $E_i$ is the $\pi$-exceptional divisor with center $\overline{D_i}$ and $\widetilde{E_i}$ is the $\pi$-strict transform of $\tau_0^{-1}(D_i)$.
	Let $S_i \coloneqq E_i \cap \widetilde{E_i}$, which is $f^{-1}$-invariant.
	Then $\pi|_{\widetilde{E_i}} \colon \widetilde{E_i} \cong \tau_0^{-1}(D_i)$ and $\pi|_{S_i} \colon S_i \cong\overline{D_i}$.
	By Theorem~\ref{thm-TID-Fano-3fold}, $\Delta_{\tau}$ is contained in some toric boundary of $Y$.
	We shall discuss case by case in terms of $\Delta_{\tau}$.

	\par \vskip 0.3pc \noindent
	\textbf{Case 1: $\Delta_{\tau}$ contains at least one section of $p$.}
	Then the $g^{-1}$-invariant $\Delta_{\tau}$ consists of either one or two disjoint sections of $p$ (cf.~Theorem \ref{thm-structure-mmp} and \cite{meng2020nonisomorphic}*{Theorem 3.3}).
	Hence, for every $z \in Z$, the fibre $Y_z \not\subset \Delta_{\tau}$.
	Then every surface $\overline{D_i}$ blown up by $\pi$ is either disjoint from $W_z$ or intersects with $W_z$ at some points;
	otherwise, $\tau_0(W_z \cap \overline{D_i})$ being a curve contradicts $Y_z \not\subset \Delta_{\tau}$.
	So $X_z$ is the blow-up of $W_z$ along several points $\bigsqcup_{i=1}^{r} \overline{D_i}\cap W_z$ and hence $X_z$ is smooth and irreducible.
	Since $X$ is Fano and $p\circ\tau$ is flat, each $X_z$ is a del Pezzo surface by the adjunction formula.
	Then $W_z$ is also Fano by the ramification divisor formula.
	So our lemma holds in this case.

	\par \vskip 0.3pc \noindent
	\textbf{Case 2: $\Delta_{\tau}$ contains no section of $p$.}
	Then $\Delta_{\tau}$ consists of prime divisors $D_i$, each of which is the pullback of a smooth rational curve on $Z$ along $p$ (cf.~\cite{meng2020nonisomorphic}*{Corollary 3.4}).
	We may assume $Y_z \subseteq D_i$ for some $i$;
	otherwise, $W_z \cong X_z$ being Fano follows from the adjunction.
	Without loss of generality, we may further assume $Y_z \subseteq D_1$ and $\pi$ factors as $X \xrightarrow{\pi_1} X_1 \to W$ where $\pi_1$ is the blow-up with $E_{\pi_1} = E_1 = \pi^{-1}(\overline{D_1})$.

	If $X_1$ is Fano, then we are done by induction on $r$.
	Thus we may assume that $X_1$ is not Fano.
	Let $H_z \coloneqq (\pi|_{\widetilde{E_1}})^{-1}(W_z) \cong W_z$ and
	\[
		\ell_z \coloneqq S_1 \cap H_z = (E_1 \cap \widetilde{E_1}) \cap H_z = (\pi|_{\widetilde{E_1}})^{-1}(W_z \cap \overline{D_1}),
	\]
	which is a cross-section of the ruling $H_z\cong W_z \to Y_z$; see the following picture.

	\begin{center}
		\begin{tikzpicture}
			\pgfmathsetmacro{\cubex}{4}
			\pgfmathsetmacro{\cubey}{2}
			\pgfmathsetmacro{\cubez}{2}
			\draw (0,0,0) -- ++(-\cubex,0,0) -- ++(0,-\cubey,0) -- ++(\cubex,0,0) -- cycle;
			\draw (0,0,0) -- ++(0,0,-\cubez) -- ++(0,-\cubey,0) -- ++(0,0,\cubez) -- cycle;
			\draw (0,0,0) -- ++(-\cubex,0,0) -- ++(0,0,-\cubez) -- ++(\cubex,0,0) -- cycle;
			\draw[dashed] (-\cubex,-\cubey,-\cubez) -- ++(\cubex,0,0);
			\draw[dashed] (-\cubex,-\cubey,-\cubez) -- ++(0,\cubey,0);
			\draw[dashed] (-\cubex,-\cubey,-\cubez) -- ++(0,0,\cubez);
			\draw[fill=yellow,opacity=0.6] (-\cubex/2,0,-\cubez) -- ++(0,0,\cubez) -- ++(0,-\cubey,0);
			\draw[fill=yellow,opacity=0.6,dashed] (-\cubex/2,0,-\cubez) -- ++(0,-\cubey,0) -- ++(0,0,\cubez);
			\draw[fill=black!20,opacity=0.5] (0,0,-\cubez/2) -- ++(-\cubex/2,0,0) -- ++(0,-\cubey,0) -- ++(\cubex/2,0,0) -- cycle;
			\node at (-\cubex*3/4,0,-\cubez*3/5) {\scriptsize $E_1$};
			\node at (-\cubex/4,0,-\cubez*3/4) {\scriptsize $\widetilde{E_1}$};
			\node at (-\cubex/2,-\cubey*3/4,-\cubez/4) {\scriptsize $S_1$};
			\node at (-\cubex/4,-\cubey/2,-\cubez/4) {\scriptsize $H_z$};
			\node at (-\cubex/2,-\cubey*2/5,-\cubez/2) {\scriptsize $\ell_z$};
			\draw (\cubex*5/4,0,0) -- ++(-\cubex/2,0,0) -- ++(0,-\cubey,0) -- ++(\cubex/2,0,0) -- cycle;
			\draw (\cubex*5/4,0,0) -- ++(0,0,-\cubez) -- ++(0,-\cubey,0) -- ++(0,0,\cubez) -- cycle;
			\draw (\cubex*5/4,0,0) -- ++(-\cubex/2,0,0) -- ++(0,0,-\cubez) -- ++(\cubex/2,0,0) -- cycle;
			\draw[dashed] (\cubex*3/4,-\cubey,-\cubez) -- ++(\cubex/2,0,0);
			\draw[fill=yellow,opacity=0.6] (\cubex*3/4,0,-\cubez) -- ++(0,0,\cubez) -- ++(0,-\cubey,0);
			\draw[fill=yellow,opacity=0.6,dashed] (\cubex*3/4,0,-\cubez) -- ++(0,-\cubey,0) -- ++(0,0,\cubez);
			\draw[fill=black!20,opacity=0.5] (\cubex*5/4,0,-\cubez/2) -- ++(-\cubex/2,0,0) -- ++(0,-\cubey,0) -- ++(\cubex/2,0,0) -- cycle;
			\node at (\cubex,-\cubey/2,-\cubez/4) {\scriptsize $W_z$};
			\node at (\cubex*3/4,-\cubey*3/4,-\cubez/4) {\tiny $\overline{D_1}$};
			\draw (\cubex*2,0,0) -- ++(0,0,-\cubez) -- ++(0,-\cubey,0) -- ++(0,0,\cubez) -- cycle;
			\draw[blue] (\cubex*2,0,-\cubez/2) -- ++(0,-\cubey,0);
			\node at (\cubex*2,-\cubey/2,-\cubez*0.47) {\tiny $Y_z$};
			\node at (\cubex*2,-\cubey*3/4,-\cubez/4) {\tiny $D_1$};
			\draw[->] (\cubex/4,-\cubey/2,-\cubez/4) -- ++(\cubex/3,0,0);
			\draw[->] (\cubex*3/2,-\cubey/2,-\cubez/4) -- ++(\cubex/3,0,0);
			\node at (\cubex*0.43,-\cubey*2/5,-\cubez/4) {\scriptsize $\pi$};
			\node at (\cubex*1.67,-\cubey*2/5,-\cubez/4) {\scriptsize $\tau_0$};
		\end{tikzpicture}
	\end{center}

	\begin{claim}\label{key-claim-fano}
		$(\ell_z \cdot S_1)_{\widetilde{E_1}} = 0$.
	\end{claim}
	Suppose Claim \ref{key-claim-fano} for the time being.
	Then we have
	$$(\ell_z^2)_{H_z} = (\ell_z \cdot S_1|_{H_z})_{H_z} = (\ell_z \cdot S_1)_{\widetilde{E_1}} = 0.$$
	Since $\ell_z$ is an (irreducible) cross-section of the ruling $H_z\to Y_z$, our $H_z \cong \bF_0$ (cf.~\cite{hartshorne1977algebraic}*{Chapter \RomanNumeralCaps{5}, Proposition 2.20}).
	As a result, $W_z \cong H_z \cong \bF_0$ and hence our lemma follows.
	So we are only left to show Claim~\ref{key-claim-fano}.

	\par \vskip 0.3pc \noindent
	\textbf{Proof of Claim~\ref{key-claim-fano}.}
	Since $X_1$ is not Fano and our $S_1\cong\overline{D_1}$ has Picard number $2$, by \cite{wisniewski1991contractions}*{Proposition 3.6}, either of the following cases occurs:
	\begin{enumerate}[label=(\arabic*), ref=(\arabic*)]
		\item \label{W91_3.6_P1xP1}
		      $S_1 \cong \overline{D_1} \cong \bF_0$ and $E_1 \cong S_1 \times \bP^1$ is a trivial bundle;
		\item \label{W91_3.6_P1-bundle}
		      $X$ admits another blow-down $\pi_1' \colon X \to X_1'$ onto a smooth fourfold $X_1'$ which contracts the divisor $E_1$ to a smooth surface $S' \subseteq X_1'$.
	\end{enumerate}

	If Case \ref{W91_3.6_P1xP1} holds, then we have $E_1 \cong \bP^1 \times \bP^1 \times \bP^1$ and hence
	\begin{align*}
		(\ell_z \cdot S_1)_{\widetilde{E_1}} = (\ell_z \cdot E_1|_{\widetilde{E_1}})_{\widetilde{E_1}} = \ell_z \cdot E_1 = -\ell_z \cdot \widetilde{E_1} = -(\ell_z \cdot \widetilde{E_1}|_{E_1})_{E_1} = -(\ell_z \cdot S_1)_{E_1} = 0,
	\end{align*}
	noting that $(E_1 + \widetilde{E_1}) \cdot \ell_z = (D_1 \cdot Y_z)_Y = 0$ by the projection formula. 
	So Claim~\ref{key-claim-fano} holds for Case \ref{W91_3.6_P1xP1}.
	From now on, we assume Case \ref{W91_3.6_P1-bundle}.
	\begin{claim}\label{claim-section-is-contracted}
		Under the condition of Case \ref{W91_3.6_P1-bundle}, $S_1$ is contracted by $\pi_1'$.
	\end{claim}

	\begin{proof}
		Suppose the contrary that $\pi_1'(S_1) = S'$.
		Note that $\rho(S') = \rho(E_1) - 1 = \rho(\overline{D_1}) = 2$, hence $S'\cong S_1$ is a (rational) ruled surface.
		Note also that the ruling of $\pi|_{E_1}$ induces a natural ruling on $S'$ since $\pi_1'$ does not contract any fibre of $\pi$.
		By Lemma~\ref{lem-equiv-MMP}, $\pi_1'$ is $f$-equivariant after iteration.
		Take a general $f|_{S'}$-periodic fibre of $S' \to \bP^1$ (cf.~\cite{fakhruddin2003questions}*{Theorem 5.1}) and denote its $\pi_1'$-inverse image in $E_1 \subseteq X$ by $T$, which is an $f$-periodic Hirzebruch surface.
		After iteration, we may assume $f(T) = T$.

		On the one hand, $T$ contains a fibre of $\pi$ since every fibre of $S'\to \bP^1$ is dominated by a fibre of $E_1 \to \overline{D_1}$.
		Hence, considering $\rho(T) = \rho(\overline{D_1}) = 2$, $\pi(T)$ is a curve.
		Since $S_1 \cong \pi(S_1) = \overline{D_1}$, we have $\pi(T) = \pi(c \coloneqq S_1 \cap T) \cong \bP^1$ on $\overline{D_1}$.
		So $\pi|_T$ gives another ruling of $T$ different from $\pi_1'|_T$;
		thus $T \cong \bF_0$.
		On the other hand, since $\pi_1'(S_1) = S'$, we see that $c$ is an $(f|_T)^{-1}$-invariant curve.
		By \cite{meng2020nonisomorphic}*{Lemma 3.1}, $c = \{\cdot\} \times \bP^1$ or $\bP^1 \times \{\cdot\}$.
		Nevertheless, this is impossible since $c$ is neither a fibre of $\pi$ nor $\pi_1'$.
	\end{proof}

	\noindent
	\textbf{End of the proof of Claim~\ref{key-claim-fano} (and Lemma~\ref{lem-fibre-fano}).}
	By Claim~\ref{claim-section-is-contracted}, $S_1$ is contracted to a curve, since the dimension of each fibre of $\pi_1'$ is no more than one.
	If $\ell_z$ is contracted, then we have $E_1 \cdot \ell_z < 0$ since $E_1$ is $\pi_1'$-anti-ample.
	Then
	\[
		(K_{E_1} \cdot \ell_z)_{E_1} = (K_X + E_1) \cdot \ell_z < 0.
	\]
	So $\pi_1'|_{E_1}$ is a Fano contraction.
	In particular, $(\ell_z \cdot S_1)_{E_1} = 0$ by the cone theorem.

	If $\ell_z$ is not contracted by $\pi_1'$, then $\overline{D_1} \cong S_1 \cong \bF_0$ since the Hirzebruch surface $S_1$ admits another ruling (induced by $\pi_1'|_{S_1}$).
	In this case, for $c'$ being a fibre of $\pi_1'|_{S_1}$, we have
	\[
		(K_{\overline{D_1}} \cdot \pi(c'))_{\overline{D_1}} = (K_{S_1} \cdot c')_{S_1} = ((K_{E_1} + S_1) \cdot c')_{E_1} = (K_{E_1} \cdot c')_{E_1},
	\]
	since $(S_1 \cdot c')_{E_1} = 0$ by the cone theorem and $\pi|_{S_1}$ is an isomorphism.
	So $E_1 \cong \overline{D_1} \times \bP^1$ (cf.~\cite{wisniewski1991contractions}*{Remark following Lemma 3.3}).
	With the same argument as in Case \ref{W91_3.6_P1xP1}, $(\ell_z \cdot S_1)_{E_1} = 0$.
	Now Claim~\ref{key-claim-fano} is proved and we have finished the proof of Lemma~\ref{lem-fibre-fano}.
\end{proof}

The following lemma contributes to showing the splitting-ness of $\cE$.
We recall Notation~\ref{notation-walls} and Lemma~\ref{lem-locally-finite-walls} for the related notations and properties.
\begin{lem}\label{lem-semistable-split}
	Suppose Notation \ref{notation-splitting} and
	the existence of the following exact sequence
	\begin{equation}\label{eq-exact-sq-stable}
		0 \to \cF_1 \to \cE \to \cQ \to 0,
		\tag{$\dagger$}
	\end{equation}
	with $\cF_1$ and $\cQ$ being invertible sheaves such that the wall $W_{\cE}(\cF_1) \neq \emptyset$.
	Then $\cE$ splits.
\end{lem}

\begin{proof}
	Since $W_{\cE}(\cF_1) \neq \emptyset$, there exists $A_0\in \textup{Amp}(Y)$ such that $\xi_{\cF_1} \cdot A_0^2 = 0$.
	By Lemma~\ref{lem-locally-finite-walls}, we can take a sufficiently small convex compact neighbourhood $\widehat{K}$ of $A_0 \in \Amp(Y)$ and let $K \subseteq P(Y)$ be its homeomorphic image such that all the walls in $K$ pass through $A_0^2$.
	Since $\xi_{\cF_1} \cdot H^2 = 0$ for each $H \in W_{\cE}(\cF_1)$, there is a chamber $\mathcal{C}$ in $K$ such that for any ample ($\bR$-Cartier) divisor $A$ with $A^2\in \mathcal{C}$, we have $\xi_{\cF_1} \cdot A^2 < 0$.
	Fix one such $A_1$.

	If $\cE$ is $A_1$-semistable, then by our dynamical assumption and \cite{amerik2003endomorphisms}*{Proposition 2.4}, $\cE$ splits and our lemma holds.

	Suppose that $\cE$ is not $A_1$-semistable.
	Let $\cF_2$ be a maximal destabilizing (saturated invertible) sheaf associated to $A_1$.
	Then one has $\xi_{\cF_2} \cdot A_1^2 > 0$ by definition; thus $\cF_2\not\subset\cF_1$.
	Consider the natural restriction of the exact sequence \eqref{eq-exact-sq-stable} to $\cF_2$.
	Since $\cF_2\not\subset\cF_1$ and $\cQ$ is locally free, the map $\cF_2\to \cQ$ is an injection.
	If $c_1(\cF_2) = c_1(\cQ)$, then it is easy to verify $\cE = \cF_1 \oplus \cF_2$, which shows the splitting-ness of our $\cE$.
	If $c_1(\cF_2) < c_1(\cQ)$, then
	\[
		\xi_{\cF_2} \cdot A_0^2 < \xi_{\cQ} \cdot A_0^2 = -\xi_{\cF_1} \cdot A_0^2 = 0 < \xi_{\cF_2} \cdot A_1^2.
	\]
	This in turn implies that $(x_0 A_0 + (1 - x_0)A_1)^2$ lies in the wall $W_{\cE}(\cF_2)$ for some $0 < x_0 < 1$.
	However, $x_0 A_0 + (1 - x_0)A_1 \in \widehat{K}$ by the convexity;
	hence $(x_0 A_0 + (1 - x_0) A_1)^2 \in K$, a contradiction to the choice of $K$, noting that $W_{\cE}(\cF_2)$ does not pass through $A_0^2$.
\end{proof}


\begin{proof}[Proof of Theorem~\ref{thm-splitting-P1-bundle-over-minimal}]
	We divide the proof into the following three cases in terms of $Z$.
	\par \vskip 0.5pc \noindent
	\textbf{Case A: $Z \cong \bP^2$.}
	We may write $Y = \bP_Z(\cF)$ where $\cF \cong \cO_Z \oplus \cO_Z(-k)$ with $k = 0, 1$ or $2$ (cf.~\cite{szurek1990fano}).
	Let $F \coloneqq p^*L$ be the ``fibre'' class of $p$ where $L$ is any line on $Z \cong \bP^2$ and $S$ the section class of $Y$ such that $\cO_Y(S)|_S \cong \cO_S(-k)$ (with respect to the surjection $\cF \to \cO_Z(-k)$).
	If $k = 1$ or $2$, then $S|_S$ is not pseudo-effective and hence $S$ is $f^{-1}$-periodic (cf.~\cite{meng2020nonisomorphic}*{Lemma 2.3});
	if $k = 0$, let $S$ be an $f$-periodic section (cf.~\cite{fakhruddin2003questions}*{Theorem 5.1}).
	In both cases, we may assume
	$S$ is $f$-invariant after iteration.
	Up to a twist, we assume $c_1(\cE) = aS + bF$ with $-1 \leq a,b \leq 0$.
	Then $\cO_{Y_z}(c_1(\cE|_{Y_z})) \cong \cO_{Y_z}(a)$ for any $z \in Z$.
	Applying $\bP_{Y_z}(\cE|_{Y_z}) = W_z \cong \bF_c$ with $c \leq 1$ by Lemma~\ref{lem-fibre-fano}, we have $\cE|_{Y_z} \cong \cO_{Y_z} \oplus \cO_{Y_z}(a)$ for any $z \in Z$.
	Thus the function $z \mapsto h^0(Y_z,\cE|_{Y_z})$ is constant and the natural morphism $ p^*p_*\cE \to \cE$ has domain a locally free sheaf, which is an evaluation map on every fibre (cf.~\cite{hartshorne1977algebraic}*{Chapter \RomanNumeralCaps{3}, Corollary 12.9}).
	Since $a\le 0$, the global sections of $\cE|_{Y_z}$ are constant.
	Then we have an exact sequence with $\cQ$ being a vector bundle.
	\begin{equation}
		\label{eq-evaluation-sequence-P2}
		0 \to p^*p_*\cE \to \cE \to \cQ \to 0
		\tag{$\ast$}
	\end{equation}

	\textbf{\textit{Suppose that $a = 0$.}}
	Then $p_*\cE$ is locally free of rank $2$;
	thus $p^*p_*\cE\cong\cE$ and $W \cong Y' \times_Z Y$ with $Y'=\bP_Z(p_*\cE)$ by the base change.
	Since $h|_{Y'}$ is int-amplified (cf.~Lemma~\ref{lem-intamplified}) and $Z \cong \bP^2$, we have $p_*\cE$ and hence $\cE$ split (cf.~\cite{amerik2003endomorphisms}*{Proposition 3}).

	\textbf{\textit{Suppose that $a = -1$.}}
	Then $p_*\cE$ is a line bundle, say $\cO_Z(e)$ for some $e \in \bZ$.
	Hence $\cF_1 \coloneqq p^*p_*\cE \cong \cO_Y(eF)$ and $\cQ \cong \cO_Y(-S + (b-e)F)$.
	Note that
	\[
		\Ext^1(\cQ, p^*p_*\cE) = H^1(Y, S + (2e-b)F) = H^1(Y, K_Y + (S - K_Y) + (2e-b)F).
	\]
	Using the relative canonical bundle formula, we have
	\begin{align*}
		S - K_Y \equiv S-(-2S + p^*(K_Z + \det\cF)) = 3S + (3 + k)F = \frac{5}{2}S + (3 + k)F + \frac{1}{2}S
	\end{align*}
	where $\cO_Y(S)|_S \cong \cO_S(-k)$ with $k = 0, 1, 2$.
	Then one can easily verify that $\frac{5}{2}S + (3+k)F$ is nef and big for any $k = 0, 1, 2$.
	Since $\frac{1}{2}S$ has the support with only normal crossings, by the Kawamata--Viehweg vanishing theorem (cf.~e.g.,~\cite{kollar1998birational}*{Theorem 2.64}), $\Ext^1(\cQ, p^*p_*\cE)$ vanishes if $t \coloneqq 2e - b \geq 0$, noting that $F$ is nef on $Y$.
	So $t \geq 0$ implies that \eqref{eq-evaluation-sequence-P2} and hence $\cE$ split.
	Therefore, we may assume that $t < 0$.
	Let $\xi_{\cF_1} \coloneqq 2c_1(p^*p_*\cE) - c_1(\cE) \sim S + tF$.
	Then our theorem for the case $Z\cong\bP^2$ follows from Claim~\ref{claim-nonempty-wall-P2} and Lemma~\ref{lem-semistable-split}.

	\begin{claim}\label{claim-nonempty-wall-P2}
		$W_{\cE}(\cF_1) \neq \emptyset$ when $t < 0$.
	\end{claim}

	\noindent
	\textbf{Proof of Claim~\ref{claim-nonempty-wall-P2}.}
	Note that up to a multiple, any ample divisor on $Y$ can be written as $D = D(u) \coloneqq S + uF$ with $u > k$.
	So we have the following:
	\begin{align*}
		\xi_{\cF_1} \cdot D(u)^2 = (S + tF) \cdot (S + uF)^2 = (u - (k - t))^2 - t^2 + kt.
	\end{align*}
	Since $k - t > k$ and $-t^2 + kt = t(k - t) < 0$, there exists $u' > k$ such that $\xi_{\cF_1} \cdot D(u')^2 = 0$ by the continuity, which completes the proof of our claim.

	\par \vskip 0.5pc \noindent
	\textbf{Case B: $Z \cong \bF_0$.}
	First, we may write $Y = \bP_Z(\cF)$ where $\cF \cong \cO_Z \oplus \cO_Z(-k_1,-k_2)$ with $0 \leq k_1, k_2 \leq 1$ (resp.~$\cO_Z \oplus \cO_Z(1, -1)$) (cf.~\cite{szurek1990fano}).
	Note that these Fano threefolds $Y$ with $\rho(Y) = 3$ has exactly $3$ extremal rays in $\NE(Y)$.
	Let $F_i \coloneqq p^*L_i$ be the ``fibre'' class of $p$ where $L_1 \cong \cO_{Z}(1, 0)$ and $L_2 \cong \cO_Z(0,1)$, and $S$ the section class of $p$ such that $\cO_Y(S)|_S \cong \cO_S(-k_1, -k_2)$ (resp.~$\cO_S(1, -1)$).
	Similar to \textbf{Case A}, we may assume our $S$ is $f$-invariant after iteration.
	Up to a twist, we assume $c_1(\cE) = aS + b_1 F_1 + b_2 F_2$ with $-1 \leq a, b_i \leq 0$.
	Then we have $\cE|_{Y_z} \cong \cO_{Y_z} \oplus \cO_{Y_z}(a)$ for any $z \in Z$ and we get the exact sequence \eqref{eq-evaluation-sequence-P2} again, noting that the global sections of $\cE|_{Y_z}$ are constant.

	\textit{\textbf{Assume first that $a = -1$.}}
	Then $p_*\cE$ is a line bundle, say $\cO_Z(e_1, e_2)$ for some $e_i \in \bZ$.
	Hence $\cF_1 \coloneqq p^*p_*\cE \cong \cO_Y(e_1 F_1 + e_2 F_2)$ and $\cQ \cong \cO_Y(-S + (b_1 - e_1)F_1 + (b_2 - e_2)F_2)$.
	Then
	\begin{align*}
		\Ext^1(\cQ, p^*p_*\cE) & = H^1(Y, K_Y + (S - K_Y) + (2e_1 - b_1)F_1 + (2e_2 - b_2)F_2).
	\end{align*}
	Applying the relative canonical bundle formula, we have
	\[
		S - K_Y \equiv S - (-2S + p^*(K_Z + \det\cF)) = 3S + (2 + k_1)F_1 + (2 + k_2)F_2 ~(\textup{resp.~} 3S + F_1 + 3F_2).
	\]
	Using the three extremal rays of $\NE(Y)$, we can verify that $S - K_Y$ (resp.~$S - K_Y - F_1$) is nef and big, noting that the bigness follows from the nefness and the positive self-intersection, and when $\cF \cong \cO_Z \oplus \cO_Z(1,-1)$, the two horizontal extremal curves lie in distinct sections of $p$.
	By the Kawamata--Viehweg vanishing theorem (cf.~e.g.,~\cite{kollar1998birational}*{Theorem 2.64}), $\Ext^1(\cQ, p^*p_*\cE)$ vanishes if $t_i \coloneqq 2e_i - b_i \geq 0$ for $i = 1, 2$ (resp.~$t_1 \geq -1$ and $t_2 \geq 0$).
	Therefore, $t_i \geq 0$ $(i = 1, 2)$ (resp.~$t_1 \geq -1$ and $t_2 \geq 0$) implies that \eqref{eq-evaluation-sequence-P2} and hence $\cE$ split.
	So we may assume $t_i < 0$ for some $i$ (resp.~$t_1 < -1$ or $t_2 < 0$).
	Let $\xi_{\cF_1} \coloneqq 2c_1(p^*p_*\cE) - c_1(\cE) \sim S + t_1 F_1 + t_2 F_2$.
	Then our theorem for the case $Z \cong \bF_0$ and $a = -1$ follows from Claim~\ref{claim-nonempty-wall-P1xP1} and Lemma~\ref{lem-semistable-split}.

	\begin{claim}\label{claim-nonempty-wall-P1xP1}
		$W_{\cE}(\cF_1) \neq \emptyset$ when $t_i < 0$ for some $i$ (resp.~$t_1 < -1$ or $t_2 < 0$).
	\end{claim}

	\noindent
	\textbf{Proof of Claim~\ref{claim-nonempty-wall-P1xP1}.}
	Recall that $Y = \bP_Z(\cF)$ with $\det\cF = \cO(-k_1, -k_2)$ with $k_i = 0$ or $1$ (resp.~$\cO(1,-1)$).
	Then up to a multiple, any ample divisor $D$ on $Y$ can be written as $D = D(u_1, u_2) \coloneqq S + u_1 F_1 + u_2 F_2$ with $u_i > k_i$ (resp.~$u_1 > 0$ and $u_2 > 1$). 
	Hence,
	\begin{align*}
		\xi_{\cF_1} \cdot D(u_1, u_2)^2
		                 & = (S + t_1 F_1 + t_2 F_2) \cdot (S + u_1 F_1 + u_2 F_2)^2              \\
		                 & = 2(u_1 - (k_1-t_1)) (u_2 - (k_2-t_2)) - 2 t_1 t_2 + t_1 k_2 + t_2 k_1 \\
		(\textup{resp.}~ & = 2(u_1 + (t_1 + 1))(u_2 + (t_2 - 1)) - 2t_1 t_2 + t_1 - t_2).
	\end{align*}
	Since $u_i > k_i$ for each $i$ (resp.~$u_1 > 0$ and $u_2 > 1$) and $t_j < 0$ for some $j$ (resp.~$t_1 < -1$ or $t_2 < 0$), there exist $u_i' > k_i$ (resp.~$u_1' > 0$ and $u_2' > 1$) such that $\xi_{\cF_1} \cdot D(u_1', u_2')^2 = 0$ by the continuity.
	So our claim holds.

	\textit{\textbf{We still need to consider the case $a = 0$.}}
	Now $p_*\cE$ is locally free of rank $2$ and one gets the following commutative diagram such that $p^*p_*\cE \cong \cE$ and $W \cong Y \times_Z Y'$.
	\[
		\xymatrix@C=3pc{
		X \ar[r] \ar[rd]^{\pi} \ar[rdd]_{\tau}	&X_{r-1} \ar[r]^{q_{r-1}} \ar[d]_{\pi_r}		&T \ar[d]^{\phi} \\
		&W = \bP_Y(\cE) \ar[r]_q \ar[d]_{\tau_0}	&Y' \coloneqq \bP_Z(p_*\cE) \ar[d]^{p'} \\
		&Y \ar[r]_p								&Z
		}
	\]
	Clearly, $q \colon W \to Y'$ is also a (smooth) $\bP^1$-bundle.
	Let $r \coloneqq \rho(X) - \rho(W)$ and $\overline{D} \coloneqq \bigsqcup_{i=1}^{r} \overline{D_i} \subseteq W$ be the blown up centres of $\pi$, each component of which is of dimension $2$.

	$1^{\circ} \colon$ If $\dim q(\overline{D_i}) = 2$ for all $1 \leq i \leq r$, then each $\overline{D_i}$, as an $h^{-1}$-invariant divisor on $q^{-1}(q(\overline{D_i}))$, is a subsection of $q$;
	hence $q \circ \pi$ is a conic bundle and hence $Y'$ is a smooth Fano threefold (cf.~Lemma~\ref{lem-base-Fano}).

	$2^{\circ} \colon$ Otherwise, after rearranging the blow-ups, we may assume that $q(\overline{D_r}) \eqqcolon C$ is a curve in $Y'$, where $\overline{D_r}$ is the blow-up centre of $\pi_r$.
	\textbf{We claim that $\overline{D_r} = q^{-1}(C)$ in this case.}
	Note that $\tau_0(\overline{D_r})$ is a divisor on $Y$ and thus $p'(C) = p \circ \tau_0(\overline{D_r})$ cannot be a point.
	Let $F = p'^{-1}(p'(C))$, a Hirzebruch surface.
	Note also that $\overline{D_r} \subseteq q^{-1}(F)$ is a (prime) divisor, and $q|_{q^{-1}(F)} \colon q^{-1}(F) \to F$ is a $\bP^1$-bundle by the base change.
	Then $q^{-1}(C) = \left(q|_{q^{-1}(F)}\right)^{-1}(C) \supseteq \overline{D_r}$ is irreducible;
	hence they coincide and our claim holds.

	Let $T$ be the blow-up of $Y'$ along the curve $C$.
	By Lemma~\ref{lem-blowup-basechange}, $X_{r-1} \cong T \times_{Y'} W$ and we denote by $q_{r-1} \colon X_{r-1} \to T$ the natural projection.
	Then $q_{r-1}$ is a $\bP^1$-bundle.
	Note that $p' \circ \phi \colon T \to Z$ is a conic bundle and $\rho(X) - \rho(X_{r-1}) = r - 1$.
	By induction, $Y'$ is dominated by a smooth Fano threefold.

	No matter $1^{\circ}$ or $2^{\circ}$ occurs, the above commutative diagram is $f$-equivariant after iteration (cf.~Lemma~\ref{lem-equiv-MMP}).
	Hence, $p_*\cE$ splits by \cite{meng2020nonisomorphic}*{Theorem 6.4} and then $\cE = p^*p_*\cE$ splits.
	This completes the proof for the case $Z \cong \bF_0$.

	\par \vskip 0.5pc \noindent
	\textbf{Case C: $Z \cong \bF_1$.}
	Denote by $c$ and $\ell$ the negative section and a fibre of the Hirzebruch surface $Z$, respectively.
	We may write $Y = \bP_Z(\cF)$ where $\cF \cong \cO_Z \oplus \cO_Z(-k(C + \ell))$ with $k = 0$ or $1$ (cf.~\cite{szurek1990fano}).
	Let $C \coloneqq p^*c$, $L \coloneqq p^*\ell$ be two ``fibre'' classes, and $S$ the section class of $p$ such that $\cO_Y(S)|_S \cong \cO_Z(-k(C + \ell))$.
	Similar to \textbf{Case A}, we may assume our $S$ is $f$-invariant after iteration.
	Up to a twist, we may assume $c_1(\cE) = aS + b_1 C + b_2 L$ with $-1 \leq a, b_i \leq 0$.
	Then, $\cO_{Y_z}(c_1(\cE|_{Y_z})) \cong \cO_{Y_z}(a)$ for any $z \in Z$, and we get the exact sequence \eqref{eq-evaluation-sequence-P2} again, noting that the global sections of $\cE|_{Y_z}$ are constant.

	\textbf{\textit{Assume that $a = -1$.}}
	Then $p_*\cE$ is a line bundle, say $\cO_Z(e_1 c + e_2 \ell)$ for some $e_i \in \bZ$.
	So $\cF_1 \coloneqq p^*p_*\cE \cong \cO_Y(e_1 C + e_2 L)$ and $\cQ \cong \cO_Y(-S + (b_1 - e_1) C + (b_2 - e_2) L)$.
	Similar to \textbf{Case B}, one can verify that
	\[
		S - K_Y \equiv S-(-2S + p^*(K_Z + \det\cF)) = 3S + (2 + k)C + (3 + k)L
	\]
	is nef and big for $k = 0, 1$ and then $\Ext^1(\cQ, p^*p_*\cE)$ vanishes if $t_i \coloneqq 2e_i - b_i \geq 0$ for $i = 1, 2$.
	So $t_i \geq 0$ $(i = 1, 2)$ implies that \eqref{eq-evaluation-sequence-P2} and hence $\cE$ split.
	Therefore, we assume $t_i< 0$ for some $i$.
	Let $\xi_{\cF_1} \coloneqq 2c_1(p^*p_*\cE) - c_1(\cE) \sim S + t_1 C + t_2 L$.
	Then our theorem for the case $Z \cong \bF_1$ and $a = -1$ follows from Claim~\ref{claim-nonempty-wall-F1} and Lemma~\ref{lem-semistable-split}.

	\begin{claim}\label{claim-nonempty-wall-F1}
		$W_{\cE}(\cF_1) \neq \emptyset$ when $t_i < 0$ for some $i$.
	\end{claim}

	\noindent
	\textbf{Proof of Claim~\ref{claim-nonempty-wall-F1}.}
	Note that any ample divisor $D$ on $Y$ can be written as $D = D(u_1, u_2) \coloneqq S + u_1 C + u_2 L$ with $u_2 > u_1 > k$ after replacing $D$ by a multiple.
	So
	\begin{align*}
		\xi_{\cF_1} \cdot D(u_1, u_2)^2
		= u_1^2 + 2(u_1 - (k - t_1)) ((u_2 - u_1) - (k - t_2)) - (k - 2t_1)(k - t_2).
	\end{align*}
	If $t_1 < 0$, taking $k < u_1 < k - t_1$ and $u_2 - u_1 \gg 1$, we have $\xi_{\cF_1}\cdot D(u_1,u_2)<0$.
	If $t_2 < 0$, taking $u_1 = k - t_2$ and $0 < u_2 - u_1 \ll k - t_2$, we have $\xi_{\cF_1} \cdot D(u_1, u_2)^2 \simeq t_2 (k - t_2) < 0$.
	In both cases, there exist $u_2' > u_1' > k$ such that $\xi_{\cF_1} \cdot D(u_1', u_2')^2 = 0$;
	hence our claim holds. 

	\textbf{\textit{The case $a = 0$}} has the same proof as in \textbf{Case B}.

	We have completed the proof of Theorem~\ref{thm-splitting-P1-bundle-over-minimal}.
\end{proof}

\section{Conic bundles over (the blow-up of) \texorpdfstring{$\bP^3$}{P3}}
\label{section-splitting-ness-to-P3}
In this section, we shall study conic bundles over $\bP^3$ or the blow-up of $\bP^3$ along a line.
Our main results are Theorems~\ref{thm-splitting-over-P3} and \ref{thm-splitting-over-P3-blowup} below.

\begin{thm}
	\label{thm-splitting-over-P3}
	Let $X$ be a smooth Fano fourfold admitting an int-amplified endomorphism $f$.
	Suppose that $X$ admits a conic bundle $\tau \colon X \to Y \cong \bP^3$, which factors as $X \xrightarrow{\pi} W \xrightarrow{\tau_0} Y$ as in Notation \ref{notation-splitting-factor}.
	Then $\tau_0$ is a splitting $\bP^1$-bundle.
\end{thm}
\begin{proof}
	By Theorem \ref{thm-structure-mmp-algebraic}, $W=\bP_Y(\cE)$ for some locally free rank 2 sheaf $\cE$ on $Y$.
	So our theorem follows from Lemmas~\ref{lem-equiv-MMP}, \ref{lem-intamplified} and \cite{amerik2003endomorphisms}*{Proposition 3}.
\end{proof}

\begin{thm}\label{thm-splitting-over-P3-blowup}
	Let $X$ be a smooth Fano fourfold admitting an int-amplified endomorphism $f$.
	Suppose that $X$ admits a conic bundle $\tau \colon X \to Y$ with $\varphi \colon Y \to Y' \cong \bP^3$ being the blow-up along a line.
	Suppose further that $\tau$ factors as $X \xrightarrow{\pi} W \xrightarrow{\tau_0} Y$ as in Notation \ref{notation-splitting-factor}.
	Then $\tau_0$ is a splitting $\bP^1$-bundle.
\end{thm}


\begin{lem}\label{lem-bundle-over-P2-Fano}
	Let $q \colon Y \to Z \cong \bP^2$ be a conic bundle from a smooth Fano threefold $Y$.
	Suppose that $q$ factors as $Y \xrightarrow{\varphi} Y' \xrightarrow{q_0} Z$, where $\varphi$ is the blow-up of a smooth threefold $Y'$ along a smooth curve $C$ and $q_0 \colon Y' \to Z$ is a splitting $\bP^1$-bundle.
	Then $Y'$ is Fano.
\end{lem}

\begin{proof}
	Suppose the contrary that $Y'$ is not Fano.
	Then $K_{Y'} \cdot C = 0$ (cf.~\cite{mori1983fano}*{Lemma 4.4 and Proposition 4.5}).
	After twisting, we may assume $Y' = \bP_Z(\cF)$ with $\cF \cong \cO_Z \oplus \cO_Z(a)$ and $a \leq 0$.
	Then by the relative canonical bundle formula,
	\[
		K_{Y'} = -2\xi + q_0^*(K_Z + \det\cF) \sim -2\xi + (a - 3) q_0^*H,
	\]
	where $H \subseteq Z$ is a line and $\xi \cong \cO_{Y'}(1)$ is the tautological divisor which is a section of $q_0$ (with respect to the surjection $\cF \to \cO_Z(a)$;
	thus $\cO_{Y'}(\xi)|_{\xi} \cong \cO_{\xi}(a)$).
	Then we have
	\[
		0 = K_{Y'} \cdot C = -2(\xi \cdot C) + (a - 3)H \cdot (q_0)_* C < -2(\xi \cdot C).
	\]
	This implies $\xi \cdot C < 0$ and hence $C \subseteq \xi$.
	Take a curve $\ell \subseteq \varphi^{-1}(\xi) \subseteq Y$ such that $\varphi_*\ell \equiv tC$ on $\xi \cong \bP^2$ for some $t \in \bZ_{> 0}$.
	Then $E_{\varphi} \cdot \ell \geq 0$ and
	\[
		K_Y \cdot \ell = \varphi^* K_{Y'} \cdot \ell + E_{\varphi} \cdot \ell \geq K_{Y'} \cdot \varphi_*\ell = (K_{Y'}|_{\xi} \cdot tC)_{\xi} = t K_{Y'} \cdot C = 0,
	\]
	This is absurd and our lemma is proved.
\end{proof}

In the rest of this section, we will focus on the proof of Theorem~\ref{thm-splitting-over-P3-blowup}.
During the proof,
we stick to Notation~\ref{notation-splitting} except that our $p$ here is a $\bP^2$-bundle.

\begin{proof}[Proof of Theorem~\ref{thm-splitting-over-P3-blowup}]
	By Lemma~\ref{lem-key-reduction} and Theorem~\ref{thm-splitting-over-P3}, we may assume $E_{\varphi} \subseteq \Supp \Delta_{\tau}$.
	Further, we have the following $f$-equivariant commutative diagram after iteration.
	\[
		\xymatrix@C=2.5pc{
		X \ar[r]^-{\pi} \ar[rd]_{\tau}	&W = \bP_Y(\cE) \ar[d]^{\tau_0} \ar@[blue][r]^{\color{blue} \psi}	&{\color{blue} W' = \bP_{Y'}(\cE')} \ar@[blue][d]_{\color{blue} \tau_0'} \\
		&Y \ar[d]_{p} \ar[r]^-{\varphi}										&{Y' \cong \bP^3} \\
		&Z \cong \bP^1
		}
	\]

	Note that there is a (smooth) $\bP^2$-bundle $p \colon Y \to Z \cong \bP^1$ and $\Delta_{\tau} \subseteq Y$ is a disjoint union of $r$ components with $r = \rho(X) - \rho(W)$.
	Since every $(f|_{Y'})^{-1}$-invariant prime divisor on $Y'$ is an (ample) hyperplane (cf.~\cite{horing2017totally}*{Corollary 1.2}), we have $r = 1$ (i.e., $\Delta_{\tau} = E_{\varphi}$), and $\pi$ is a single blow-up along a surface $\overline{S} \subseteq W$ (cf.~\cite{cascini2020polarized}*{Lemma 7.5}).

	For each $z\in Z$, the fibre $X_z$ is a (smooth) Fano threefold and $\pi|_{X_z}$ is the blow-up of $W_z=\bP_{Y_z}(\cE|_{Y_z})$
	along a (smooth) rational curve $C_z \coloneqq W_z|_{\overline{S}}$.
	Take an $f|_Z$-periodic point $z_0 \in Z$ (cf.~\cite{fakhruddin2003questions}*{Theorem 5.1}).
	After iteration, $f$ restricts to an int-amplified endomorphism on $W_{z_0} \cong \bP_{Y_{z_0}}(\cE|_{Y_{z_0}})$.
	Hence, it follows from \cite{amerik2003endomorphisms}*{Proposition 3} that $\cE|_{Y_{z_0}}$ splits, noting that $Y_{z_0} \cong \bP^2$.
	By Lemma~\ref{lem-bundle-over-P2-Fano}, $W_{z_0}$ is Fano.
	Note that $(p \circ \tau_0)|_{\overline{S}}$ is flat onto a smooth curve $Z$.
	Then the $C_z$'s are numerically equivalent on $\overline{S}$ and hence on $W$.
	Now for every $z \in Z$, applying the adjunction, we have
	\[
		(K_{W_z} \cdot C_z)_{W_z} = (K_W|_{W_z} \cdot C_z)_{W_z} = K_W \cdot C_z = K_W \cdot C_{z_0} = (K_{W_{z_0}} \cdot C_{z_0})_{W_{z_0}} < 0.
	\]
	Therefore, every $W_z$ is Fano (cf.~\cite{mori1983fano}*{Lemma 4.4 and Proposition 4.5}).
	On the other hand, Fano $\bP^1$-bundles over $\bP^2$ cannot deform to each other (cf.~e.g.,~\cite{mori1983fano}).
	As a result, for all $z \in Z$, $W_z \cong \bP_{Y_z}(\cE|_{Y_{z}}) \cong \bP_{Y_z}(\cO \oplus \cO(-k))$ with $k = 0, 1$ or $2$ (cf.~\cite{szurek1990fano}).

	Let $F \cong \bP^2$ and $S \cong \bF_0$ be the fibre class of $p$ and the exceptional divisor of $\varphi$, respectively.
	Up to a twist, we can write $c_1(\cE) = aS + bF$ with $-1 \leq a, b \leq 0$.
	Then $\cO_{Y_z}(c_1(\cE|_{Y_z})) \cong \cO_{Y_z}(a)$.
	By assumption, $S = E_{\varphi} = \Delta_{\tau}$; hence $\tau^*S = E + \widetilde{E}$ with $E = E_{\pi}$ being $\pi$-exceptional.
	Now we refer to Lemma~\ref{lem-key-reduction} and use the notations therein.
	By the projection formula, $(E + \widetilde{E}) \cdot \ell_{\chi} <0$, where $\ell_{\chi} \cong \bP^1$ is a fibre of $\chi$ lying in $R_{\chi}$.
	Since $\chi$ is a divisorial contraction, the locus of $R_{\chi}$ either equals $E$ or equals $\widetilde{E}$.
	In both cases, $E \cap \widetilde{E}$ is contracted to a curve (cf.~Claim~\ref{claim-section-is-contracted}).
	So we may assume $\ell_{\chi}\subseteq E \cap \widetilde{E}$.

	If the locus of $R_{\chi}$ equals $\widetilde{E}$, then $\widetilde{E} \cdot \ell_{\chi} = -1$ since $\chi$ is a blow-up.
	Since $\tau|_{\widetilde{E} \cap E}$ is an isomorphism, for a fibre $\ell_{\varphi}$ of $\varphi$, we have
	\[
		(E + \widetilde{E}) \cdot \ell_{\chi} = \tau^* S \cdot\ell_\chi= S \cdot \ell_{\varphi} = -1.
	\]
	So $E \cdot \ell_{\chi} = 0$ and thus $K_W\cdot\ell_{\psi}=K_X\cdot\ell_\chi<0$.
	Then, $\psi$ is a $K_W$-negative contraction.
	By Lemma~\ref{lem-key-reduction-W'} (or \cite{casagrande2008quasi}*{Theorem 3.14} for $\tau_0' \circ \psi = \varphi \circ \tau_0$), we have $\cE = \varphi^* \cE'$.
	Since $f|_{W'}$ is int-amplified, $\cE'$ and hence $\cE$ split (cf.~\cite{amerik2003endomorphisms}*{Proposition 3}).

	Thus we may assume that the locus of $R_{\chi}$ equals $E$, $E \cdot \ell_{\chi} = -1$, and then $\widetilde{E} \cdot \ell_{\chi} = 0$.
	Then $\chi|_{\widetilde{E}}$ induces a divisorial contraction contracting $E \cap \widetilde{E}$, noting that $K_{\widetilde{E}} \cdot \ell_{\chi} = K_X \cdot \ell_{\chi} < 0$.
	Let $H \coloneqq \pi_*^{-1}(\tau_0^{-1}\ell_{\varphi}) \subseteq \widetilde{E}$ be the $\bP^1$-bundle over $\ell_{\varphi} \cong \bP^1$. 
	Since $H \cap (E \cap \widetilde{E})$ is contracted by $\chi$, our $\chi|_H$ is a (smooth) blow-down on $H$;
	hence $\bP_{\ell_{\varphi}}(\cE|_{\ell_{\varphi}}) = \tau_0^{-1}(\ell_{\varphi}) \cong H \cong \bF_1$.
	Now $c_1(\cE)|_{\ell_{\varphi}} = (aS + bF)|_{\ell_{\varphi}} = b - a$ and $-1 \leq a, b \leq 0$ imply $b - a = \pm 1$.

	Since $W_z=\bP_{Y_z}(\cE|_{Y_{z}}) \cong \bP_{Y_z}(\cO \oplus \cO(-k))$ with $k = 0, 1$ or $2$, there are three cases:
	\begin{enumerate}
		\item $a = -1, b = 0, c_1(\cE) = -S$ and $\cE|_{Y_z} = \cO \oplus \cO(-1)$ for all $z \in Z$;
		\item $a = 0, b = -1, c_1(\cE) = -F$ and $\cE|_{Y_z} = \cO \oplus \cO$ for all $z \in Z$;
		\item $a = 0, b = -1, c_1(\cE) = -F$ and $\cE|_{Y_z} = \cO(1) \oplus \cO(-1)$ for all $z \in Z$.
	\end{enumerate}
	In Case (3), replacing $\cE$ by $\cE \otimes \cO_Y(-S)$, we may assume
	\begin{enumerate}[label=($\arabic*'$), start=3]
		\item $a = -2, b = -1, c_1(\cE) = -2S - F$ and $\cE|_{Y_z} = \cO \oplus \cO(-2)$ for all $z \in Z$.
	\end{enumerate}

	In all the cases, the function $z \mapsto h^0(Y_z,\cE|_{Y_z})$ is constant on $Z$.
	With the same argument as in the proof of Theorem~\ref{thm-splitting-P1-bundle-over-minimal}, global sections of $\cE|_{Y_z}$ are constant in the cases: (1), (2) and ($3'$).
	Then we get the same exact sequence \eqref{eq-evaluation-sequence-P2}.

	If Case (2) occurs, then $\cE|_{Y_z} \cong \cO \oplus \cO$ for all $z$;
	hence $p_*\cE$ is locally free of rank $2$ (on $\bP^1$) and splits always, which implies that $\cE \cong p^*p_*\cE$ splits.

	If Cases (1) or ($3'$) occur, then $p_*\cE$ is a line bundle, say $\cO_Z(e)$ for some $e \in \bZ$.
	Then $\cF_1 \coloneqq p^*p_*\cE \cong \cO_Y(eF)$ and $\cQ \cong \cO_Y(aS + (b - e)F)$.
	Note that
	\[
		\Ext^1(\cQ, p^*p_*\cE) = H^1(Y, -aS + (2e - b)F) = H^1(Y, K_Y + (-aS - K_Y) + (2e - b)F).
	\]
	Since $(-aS - K_Y) + (2e - b)F \equiv (3 - a)S + (4 - b)F + 2eF$, it is nef and big (for both cases (1) and ($3'$)) when $e \geq 0$.
	By the Kawamata--Viehweg vanishing theorem (cf.~e.g.,~\cite{kollar1998birational}*{Theorem 2.64}), $\Ext^1(\cQ, p^*p_*\cE)$ vanishes if $e \geq 0$.
	So $e \geq 0$ implies that \eqref{eq-evaluation-sequence-P2} and hence $\cE$ split.
	Therefore, we may assume $e < 0$.
	Then $\xi_{\cF_1} \coloneqq 2 c_1(p^*p_*\cF_1) - c_1(\cE) \sim -aS + (2e - b)F$.
	Since $e < 0$, no matter (1) or ($3'$) occurs, $2e - b < 0$.
	So our theorem follows from Claim~\ref{claim-nonempty-wall-P3} and Lemma~\ref{lem-semistable-split}.
	\begin{claim}\label{claim-nonempty-wall-P3}
		$W_{\cE}(\cF_1) \neq \emptyset$ when $t \coloneqq 2e - b < 0$.
	\end{claim}

	\noindent
	\textit{\textbf{Proof of Claim~\ref{claim-nonempty-wall-P3}.}}
	Note that any ample divisor $D$ on $Y$ can be written as $D = D(u) \coloneqq S + uF$ with $u > 1$ after replacing $D$ by a multiple.
	So we have the following:
	\begin{align*}
		\xi_{\cF_1} \cdot D(u)^2
		= -aS^3 + (-2au + t) S^2 \cdot F = -2a(u - 1) + t.
	\end{align*}
	The last equality is due to $\cO_Y(S)|_S \cong \cO_S(-1, 1)$, noting that for a fibre $C_0$ of $p|_S$ (which is the second ruling of $S$ different from $\varphi|_S$), we have $S \cdot C_0 = 1$ by the ramification divisor formula of $\varphi$ and the adjunction formula.
	Since $-2a > 0$ and $t < 0$, there exists a real number $u_0 > 1$ such that $\xi_{\cF_1} \cdot D(u_0)^2 = 0$.
	So our claim holds.
\end{proof}

%
%
%
%

\section{Splitting-ness of algebraic \texorpdfstring{$\bP^1$}{P1}-bundles}
\label{section-splittingness}

Combining with the results in Sections \ref{section-splitting-ness-to-surface} and \ref{section-splitting-ness-to-P3}, our main focus in this section is Theorem~\ref{thm-splitting}, which is a key to proving Theorem~\ref{thm-main}.
\begin{thm}\label{thm-splitting}
	Let $X$ be a smooth Fano fourfold admitting an int-amplified endomorphism $f$.
	Suppose that $\tau \colon X \to Y$ is a conic bundle which factors as $X \xrightarrow{\pi} W \xrightarrow{\tau_0} Y$ as in Notation \ref{notation-splitting-factor}.
	If one of the following holds, then $\tau_0$ is a splitting $\bP^1$-bundle.
	\begin{enumerate}[label=(\arabic*), ref=(\arabic*)]
		\item $\rho(X)-\rho(W) \neq 1$, i.e., either $\tau = \tau_0$ is elementary or $\rho(X)-\rho(W) \geq 2$;
		\item $Y \cong \bP^3$ or the blow-up of a line on $\bP^3$;
		\item $Y$ is a $\bP^1$-bundle over $\bP^2$, $\bF_0$ or $\bF_1$.
	\end{enumerate}

	Furthermore, if $\tau_0$ is not a splitting $\bP^1$-bundle, then $(\ddagger)$ there exists another conic bundle $\widehat{\tau} \colon X \to \widehat{W} \to \widehat{Y}$ with $\widehat{\tau}_0 \colon \widehat{W} \to \widehat{Y}$ being elementary such that $\rho(X)-\rho(\widehat{W}) = 2$.
	In particular, $\widehat{W}$ is a splitting $\bP^1$-bundle over $Y$.
\end{thm}

\begin{rem}\label{rem-before-split}
	\begin{enumerate}[label=(\arabic*)]
		\item We assume Notation \ref{notation-splitting} (except $p$) throughout this section.
		\item \label{rem-before-split-2} The first part of Theorem~\ref{thm-splitting} follows from Theorems~\ref{thm-splitting-P1-bundle-over-minimal}, \ref{thm-splitting-over-P3}, \ref{thm-splitting-over-P3-blowup}, and \ref{thm-splitting-W=X}, \ref{thm-splitting-r-geq-2}.
		\item We are unable to show the splitting-ness of $\tau_0$ when $\pi$ is a single blow-up, since Lemma~\ref{lem-key-reduction} cannot be applied if $E_{\varphi}\subseteq\Delta_{\tau}$;
		      see Remark~\ref{rem-reduction-lem}.
	\end{enumerate}
\end{rem}

\begin{thm}\label{thm-splitting-W=X}
	Suppose that $X = W$ (which is Fano).
	Then $\tau_0$ is a splitting $\bP^1$-bundle.
\end{thm}

\begin{proof}
	By Theorem \ref{thm-threefold-class} and Lemmas~\ref{lem-Z-to-minimal}, \ref{lem-Y-to-minimal}, we only need to consider when $Y$ is $\bP^3$ or a $\bP^1$-bundle over $\bP^2$ or $\bF_0$.
	This is clear by Theorems~\ref{thm-splitting-over-P3} and \ref{thm-splitting-P1-bundle-over-minimal}.
\end{proof}




\begin{thm}\label{thm-splitting-r-geq-2}
	Suppose that $\rho(X)-\rho(W) \geq 2$.
	Then $\tau_0$ is a splitting $\bP^1$-bundle.
\end{thm}

\begin{proof}
	If $\rho(X)-\rho(W) \geq 3$, then $X$ is a product of del Pezzo surfaces and thus our theorem follows (cf.~Lemma~\ref{lem-Rom4.2}).
	So we may assume that $\rho(X)-\rho(W) = 2$ and $X$ does not split.
	By Lemma~\ref{lem-Rom4.2-=3}, $Y$ is a smooth $\bP^1$-bundle over a del Pezzo surface $Z$.
	By Lemmas~\ref{lem-Z-to-minimal} and \ref{lem-key-reduction}, we have the following $f$-equivariant commutative diagram after iteration,
	\begin{align}\label{diagram-case(b)}
		\xymatrix@C=5pc{
		X \ar[r]^{\chi} \ar@/_2pc/[dd]_{\tau} \ar[d]^{\pi} \ar[dr]^{\eta} & X' \ar[d]_{\pi'} \ar@/^2pc/[dd]^{\tau'} \\
		W \ar[d]^{\tau_0} \ar[r]_{\psi}                                   & W' \ar[d]_{\tau_0'}                     \\
		Y \ar[r]_{\varphi} \ar[d]^p                                       & Y' \ar[d]_{p'}                          \\
		Z \ar[r]_{\phi}                                                   & Z'
		}
	\end{align}
	where $\phi$ is the blow-up of a point on a Fano surface $Z'$ and $\varphi$ is the blow-up of a smooth curve on a smooth Fano threefold $Y'$ with the exceptional locus $E_{\varphi} = p^* E_{\phi}$.
	By Theorem~\ref{thm-splitting-P1-bundle-over-minimal}, we may further assume $Z$ is a toric Fano surface with degree $6$ or $7$ (cf.~\cite{nakayama2002ruled}*{Theorem 3}).

	We claim that $E_{\varphi}$ and $\Delta_{\tau}$ have no common component.
	Suppose the contrary that there exists a surface $S$ which is a component of both $E_{\varphi}$ and $\Delta_{\tau}$.
	Note that $S$ is a $\bP^1$-bundle over $C \coloneqq \varphi(S)$ and hence $\rho(S) = 2$.
	Consider the proper transform $D \coloneqq \pi_*^{-1}(\tau_0^* S)$, which is a prime divisor of Picard number $3$ on $X$.
	Note also that $\rho(X) = \rho(Z) + 4 \geq 7$.
	Then the Lefschetz defect $\delta_X \geq 4$ and hence $X \cong S_1 \times S_2$ with $S_i$ being del Pezzo surfaces (cf.~Lemma \ref{lem-leftschetz-defect}), contradicting our assumption at the beginning of the proof.

	By Lemma~\ref{lem-key-reduction}, $\tau'$ is a Fano conic bundle, $W' = \bP_{Y'}(\cE')$ and $\cE = \varphi^*\cE'$;
	thus our theorem follows from Theorem~\ref{thm-splitting-P1-bundle-over-minimal} \textbf{Case A: $Z \cong \bP^2$} (and the induction on $\rho(Z)$).
\end{proof}


\begin{proof}[Proof of Theorem~\ref{thm-splitting}]
	By Remark \ref{rem-before-split} \ref{rem-before-split-2}, we only need to show the second part.
	Since $\tau_0$ is not a splitting $\bP^1$-bundle, we have $\rho(X)-\rho(W) = 1$ (and hence $\Delta_{\tau}$ is irreducible) by the first part.
	In the light of Theorems~\ref{thm-threefold-class}, \ref{thm-splitting-P1-bundle-over-minimal}, \ref{thm-splitting-over-P3}, \ref{thm-splitting-over-P3-blowup} and Lemma~\ref{lem-base-Fano}, we may further assume either of the following.
	\begin{enumerate}[label=(\alph*)]
		\item $Y$ admits a non-elementary conic bundle $Y \to Z$ to a smooth Fano surface $Z$;
		\item $Y$ is a smooth $\bP^1$-bundle over $Z$ where $Z$ is a del Pezzo surface with degree $6$ or $7$.
	\end{enumerate}
	Note that, in both cases, $Z$ is a toric Fano surface (cf.~\cite{nakayama2002ruled}*{Theorem 3}).

	\par \vskip 0.5pc \noindent
	\textbf{\textit{Suppose that (a) occurs.}}
	We will use \cite{mori1983fano}*{Proposition 9.10}.
	Then $\rho(Z) \leq 2$ and we have two small cases: (i) $\Delta_{\tau}$ is not $\varphi$-exceptional, or (ii) $\Delta_{\tau}$ is $\varphi$-exceptional.

	If (i) occurs, Lemma~\ref{lem-key-reduction} and Theorem~\ref{thm-splitting-P1-bundle-over-minimal} would imply $\cE$ splits, a contradiction.

	If (ii) occurs, then we have the following commutative diagram
	\begin{align}
		\xymatrix{
		X \ar[r]_{\pi} \ar@/^1pc/[rr]^{\tau} & W \ar[r]_{\tau_0} & Y \ar[r]_{\varphi} \ar@/^1pc/[rr]^{p} & Y' \ar[r]_{p_0} & Z
		}
	\end{align}
	where $Y'$ is a smooth Fano threefold, $Z\cong\bP^2$, $\bF_0$ or $\bF_1$, $\tau$ and $p$ are singular conic bundles, and $\tau_0$ and $p_0$ are smooth $\bP^1$-bundles (cf.~Theorem \ref{thm-structure-mmp} and \cite{meng2020nonisomorphic}*{Theorem 6.2}).

	\textbf{We claim that the Lefschetz defect $\delta_X \geq 3$.}
	Let $D_{Y'}$ be a section of $p_0$ and disjoint from $\varphi(\Delta_{\tau})$ (cf.~Theorem \ref{thm-threefold-class} \textbf{(B)}).
	Then the irreducible divisor $\tau^*\varphi_*^{-1} D_{Y'}$ has Picard number $\rho(Z) + 1$.
	On the other hand, the inequality $\rho(X) \geq \rho(Z) + 4$ implies that $\delta_X \geq \rho(X) - \rho(\tau^*\varphi_*^{-1} D_{Y'}) \geq 3$;
	hence our claim holds.

	By Lemma \ref{lem-leftschetz-defect}, there exists another Fano conic bundle $X \to \widehat{Y}$ such that $\rho(X)-\rho(\widehat{Y}) = 3$.
	Together with Theorems~\ref{thm-structure-mmp} and \ref{thm-splitting-r-geq-2}, $(\ddagger)$ holds.

	\par \vskip 0.5pc \noindent
	\textbf{\textit{Suppose that (b) occurs.}}
	Then we also have two small cases: (i') $\Delta_{\tau}$ is a section of $p \colon Y \to Z$, or (ii') $\Delta_{\tau}$ is the pullback of an $(f|_Z)^{-1}$-invariant curve on $Z$ (cf.~Corollary~\ref{cor-splitting-to-toric}).

	If (i') happens, then we have the commutative diagram (\ref{diagram-case(b)}).
	Since $\Delta_{\tau}$ is the section of $p$, it is not $\varphi$-exceptional.
	Hence $X' \to Y'$ is a Fano conic bundle, $W \cong W' \times_{Y'} Y$ and $\varphi(\Delta_{\tau})$ is a section of $Y' \to Z'$ (cf.~Lemma~\ref{lem-key-reduction}).
	After one more reduction if necessary, we may assume $Z' \cong \bP^2$.
	So $\cE'$ and hence $\cE$ split by Theorem~\ref{thm-splitting-P1-bundle-over-minimal}, a contradiction.

	If (ii') happens, then choosing a rational curve $C$ disjoint from $p(\Delta_{\tau})$ in $Z$ (this is doable by considering $Z$ as blow-ups from $\bP^2$), we have $\rho(X) - \rho(\tau^*p^*C) = \rho(Z) \geq 3$.
	Applying Lemma \ref{lem-leftschetz-defect} again, there exists another Fano conic bundle $X \to \widehat{Y}$ such that $\rho(X) - \rho(\widehat{Y}) = 3$.
	By Theorem~\ref{thm-splitting-r-geq-2}, we have $(\ddagger)$.
\end{proof}

\section{Proof of Theorem~\ref{thm-main} and Corollary~\ref{cor-splitting}}
\label{section-proof-of-main-thm}
To prove Theorem~\ref{thm-main}, we begin with the theorem below.
For its proof, we follow the idea of \cite{meng2020nonisomorphic}*{Proof of Theorem 8.1}.

\begin{thm}\label{thm-main-int-to-toric}
	Let $X$ be a smooth Fano fourfold admitting a conic bundle $\tau \colon X \to Y$.
	Suppose that $X$ admits an int-amplified endomorphism $f$.
	Then $X$ is toric.
\end{thm}

\begin{proof}
	We follow the notations in Theorem~\ref{thm-structure-mmp}.
	If $\rho(X)-\rho(Y)= 1$, then $X$ is a splitting $\bP^1$-bundle over the smooth toric Fano threefold $Y$ (cf.~Lemma~\ref{lem-base-Fano}, Theorems~\ref{thm-splitting-W=X} and \ref{thm-threefold-class}); thus $X$ is toric by \cite{meng2020nonisomorphic}*{Proposition 2.9}.
	If $\rho(X)-\rho(Y) \geq 4$, then $X$ is a product of del Pezzo surfaces (cf.~Lemma~\ref{lem-Rom4.2-geq-4}); hence $f = f_1 \times f_2$ after iteration, noting that $\NE(X)$ is rational polyhedral.
	So it follows from Lemma~\ref{lem-intamplified} and \cite{nakayama2002ruled}*{Theorem 3} that both $S_i$ are toric.
	As a result, $X$ is also toric.
	In the following, we may assume $r \coloneqq \rho(X) - \rho(Y) - 1 = 1$ or $2$.
	By Theorem~\ref{thm-structure-mmp}, $\tau^{-1}(D_i) = E_i \cup \widetilde{E_i}$, and $E_i$, $\widetilde{E_i}$ are (smooth) $\bP^1$-bundles over $D_i$, with $E_i \cap \widetilde{E_i}\cong D_i$ being an $f^{-1}$-invariant surface.

	\par \vskip 0.3pc \noindent
	\textbf{Case: $r = 2$.} By Theorem~\ref{thm-splitting-r-geq-2}, $X_r = \bP_Y(\cE)$ is a splitting $\bP^1$-bundle over $Y$;
	hence after a suitable twisting, we may assume $\cE = \cO_Y \oplus \cL$ with $\cL$ being trivial or not pseudo-effective.
	We claim that there is an $f^{-1}$-invariant section $S$ of $\tau$ dominating $Y$.
	Note that $\pi(E_1 \cap \widetilde{E_1})$ is an $f_{r}^{-1}$-invariant subsection over $Y$ (not contracted by $\tau_{0}$) and $\pi(E_1 \cup \widetilde{E_1}) = \tau_{0}^{-1}(D_1)$ is a $\bP^1$-bundle over $D_1$.
	If $\tau_0$ is a trivial bundle so that $X_{r} = Y \times Z \cong Y \times \bP^1$, then $\pi(E_1 \cup \widetilde{E_1}) \cong D_1 \times \bP^1$.
	Now $\pi(E_1\cap \widetilde{E_1})$ is contained in an $f_{r}^{-1}$-invariant horizontal section $S_{r}$ of $\tau_0$ (cf.~\cite{cascini2020polarized}*{Lemma 7.5} and Lemma~\ref{lem-split-pdt} applied to $\pi(E_1 \cup \widetilde{E_1})$).
	If $\tau_0$ is not a trivial bundle, then some section $S_{r}$ has $S_{r}|_{S_{r}} \cong \det \cE = \cL$ being not pseudo-effective and hence is $f_{r}^{-1}$-invariant after iteration by \cite{meng2020nonisomorphic}*{Lemma 2.3}.

	In both cases, let $S \subseteq X$ be the proper transform of $S_{r}$.
	Note that the surfaces in $W$ blown up by $\pi$ are either contained in $S_r$ or disjoint from $S_r$.
	Indeed, if there is some $\overline{D_i}$ intersects $S_r$ along a curve $C$, which is also $f_r^{-1}$-invariant,
	then the three $f_r^{-1}$-invariant prime divisors $\overline{D_i}$, $\tau_0^{-1}(\tau_0(C))$ and $\tau_0^{-1}(\tau_0(\overline{D_i})) \cap S_r$ in the $f_r^{-1}$-invariant smooth projective threefold $\tau_0^{-1}(\tau_0(\overline{D_i}))$ will have a common intersection curve $C$, a contradiction to Lemma~\ref{thm-bh}.
	So we get the $f^{-1}$-invariant section $S \cong S_{r}$ of $\tau$.


	Since $\pi(S) = Y$, we have $(E_i \cup \widetilde{E_i}) \cap S \neq \emptyset$.
	Hence we may assume $S\cap E_i \neq \emptyset$ for each $i$.
	By Theorem~\ref{thm-structure-mmp} (cf.~\cite{romano2019non}*{Remark 3.6}), after iteration, there is an $f$-equivariant birational morphism $\pi' \colon X \to X'$ over $Y$ contracting all the $E_i$ with $f' \coloneqq f|_{X'}$, such that:
	\begin{enumerate}
		\item the induced morphism $\tau_0' \colon X' \to Y$ is an algebraic $\bP^1$-bundle;
		\item $\pi'$ is the blow-up of $X'$ along $2$ smooth projective surfaces $D_i' \coloneqq F_i' \cap S'$, where $F_i' \coloneqq (\tau_0')^{-1}(D_i)$ and $S' \coloneqq \tau_0'(S)$ are $(f')^{-1}$-invariant prime divisors.
	\end{enumerate}

	Since $\rho(X) - \rho(X') = 2$, it follows from Theorem~\ref{thm-splitting-r-geq-2} that $X'$ is a splitting $\bP^1$-bundle over $Y$;
	thus $X'$ is toric (cf.~\cite{meng2020nonisomorphic}*{Proposition 2.9}).
	By Corollary~\ref{cor-splitting-to-toric}, there is a toric pair $(X', \Delta')$ such that $\Delta'$ contains all the $(f')^{-1}$-invariant prime divisors (including $F_i', S'$).
	By the construction, $\pi'$ is the composition of toric blow-ups of the intersection of prime divisors in the toric boundary starting from $(X', \Delta')$.
	Thus $X$ is toric.

	\par \vskip 0.3pc \noindent
	\textbf{Case: $r = 1$.}
	By Theorem~\ref{thm-structure-mmp}, $\tau$ factors as $\tau_0 \colon X_1 \to Y$ (resp.~$\widetilde{\tau}_0 \colon \widetilde{X_1} \to Y$).
	If both $\tau_0$ and $\widetilde{\tau}_0$ are splitting $\bP^1$-bundles, then we are done with the same argument as above.
	Otherwise, by Theorem~\ref{thm-splitting}, there exists another conic bundle $X \to \widehat{Y}$ such that $\rho(X)-\rho(\widehat{Y}) = 3$.
	So we reduce to \textbf{Case: $r = 2$} and our theorem holds.
\end{proof}

\begin{proof}[Proof of Theorem~\ref{thm-main}]
	Clearly, (2) implies (3).
	By Theorem~\ref{thm-main-int-to-toric}, (3) implies (1).
	Since every projective toric variety has a polarized endomorphism (cf.~\cite{nakayama2002ruled}*{Lemma 4} and \cite{meng2020rigidity}*{Proof of Theorem 1.4}), (1) implies (2).
\end{proof}

\begin{proof}[Proof of Corollary~\ref{cor-splitting}]
	By Theorem~\ref{thm-structure-mmp}, $W = \bP_Y(\cE)$ for some locally free sheaf $\cE$ of rank $2$ on $Y$.
	By Theorem~\ref{thm-main}, $X$ is toric.
	Fix a toric action $G$ on $X$ such that it descends to a unique action $G_Y$ on $Y$ which is $\tau$-equivariant (cf.~\cite{brion2011automorphism}*{Proposition 2.1}).
	Since $X$ is toric and $X \to W$ has connected fibres, by \cite{brion2011automorphism}*{Proposition 2.1} again, $W$ is toric with the (uniquely descended) toric action $G_W$ on $W$.
	By \cite{druel1999structures}*{Lemma 1}, $\cE$ splits.
\end{proof}


\end{document}